\documentclass[11pt]{article}
\usepackage{amsthm,amsfonts,amssymb,amsmath,mathrsfs}
\usepackage{enumerate}
\usepackage[colorlinks=true]{hyperref}
\usepackage[papersize={8.2in,11in},textwidth=16cm,textheight=22cm,centering]{geometry}
\date{}
\newcommand{\keywords}[1]{\par\smallskip\noindent\textbf{Keywords and phrases. }#1\par\smallskip}
\newcommand{\subjclass}[1]{\par\smallskip\noindent\textbf{2020 Mathematics Subject Classification. }#1\par\smallskip}
\newtheorem{Theorem}{Theorem}[section]
\newtheorem{Lemma}[Theorem]{Lemma}
\newtheorem{Corollary}[Theorem]{Corollary}
\newtheorem{Proposition}[Theorem]{Proposition}
\newtheorem{Remark}[Theorem]{Remark}
\newtheorem{Example}[Theorem]{Example}
\newtheorem{Definition}[Theorem]{Definition}

\def\deg{\operatorname{deg}}
\def\Deg{\operatorname{Deg}}

\def\codim{\operatorname{codim}}
\def\reg{\mathrm{reg}} 

\def\ri{\operatorname{ri}}
\def\rem{\operatorname{rem}}

\def\gcd{\mathrm{gcd}}

\def\mfr{{\frak m}}

\def\ebf{{\mathbf e}}

\def\ubf{{\mathbf u}}

\def\wbf{{\mathbf w}}
\def\Wbf{{\mathbf W}}

\def\mfr{{\mathfrak m}}

\def\Lcal{{\mathcal L}}

\def\Acal{{\mathcal A}}

\def\Pset{{\mathbb P}}
\def\Zset{{\mathbb Z}}

\def\Nset{{\mathbb N}}
\def\Iset{{\mathbb I}}

\theoremstyle{remark}

\theoremstyle{definition}

\begin{document}

\title{\textsc{New bounds on Castelnuovo--Mumford regularity of monomial curves and application to sumsets}}
\author{Le Tuan Hoa and Doan Quang Tien}

\maketitle
\begin{abstract}  
A monomial curve $C$ is defined by a sequence of coprime integers $0 = a_0 < a_1 < \cdots < a_k =: d$. One gap of this sequence is $a_{i+1} - a_i  - 1$. Gruson--Lazarsfeld--Peskine bound (1983) says that $\reg(C) \le d  -  k +2$, which is equal to the sum of all gaps plus 2. L'vovsky (1996) showed that it is enough  to take the sum of two largest gaps plus 2. In this paper, under some specific conditions, we give several new bounds which are better than L'vovsky’s bound.  Our method relies on the study of Ap\'ery sets  and Frobenius numbers. From this we can give new criteria to check the (arithmetically) Cohen--Macaulay and Buchsbaum property of $C$. Algorithms are provided to check these properties as well as to compute $\reg(C)$ and other invariants. We also give an application to study the structure of sumsets.
\end{abstract}

\section{Introduction}\label{Intr}
A projective monomial curve $C := C(a_1,\ldots,a_k) \subset \Pset_K^k$ is defined by  a sequence of coprime integers $0 < a_1 < \cdots < a_k =: d$. From the general result of \cite{GLP} on the Castelnuovo-Mumford regularity of curves, we already know that $\reg(C) \leq d-k+2$. Because of the importance of the Castelnuovo--Mumford regularity,  it is still of great interest to find  better bounds for  $\reg(C)$. On one side, $C$ is so special, that one can hope to get a better bound.  A gap of the above sequence is $a_{i+1} -  a_i  - 1$ for $i=0,1,\ldots, k-1$, where $a_0:= 0$. Then $d-k$ is the sum of all gaps.  L'vovsky \cite{Lv} in 1996 proved that it is enough to take the largest and second largest gaps  $\lambda_{\max}$ and $\lambda_{\mathrm{sl}}$: $\reg(C) \leq \lambda_{\max} + \lambda_{\mathrm{sl}} + 2$. This is a rather good bound. However, under additional assumptions, one can provide much better bounds, or even compute it, see,  e.g., \cite{BCFH, HeHS, LT}. Our first goal is to provide further bounds for some rather broader classes of monomial curves.

On the other hand, the coordinate ring of $C(a_1,\ldots,a_k)$  is isomorphic  to the semigroup ring $K[S]$, where $S\subseteq \Nset^2$ is generated by elements $(a_i, d- a_i)$, $i=0,1,\ldots,k$ ($a_0 := 0$) -- a combinatorial subject. It is therefore natural to ask if one can give a combinatorial proof for the above mentioned results?  A first combinatorial proof of Gruson--Lazarsfeld--Peskine bound for $C$ was provided only in 2014, see  \cite{Nit}. Recently, it was shown in \cite[Section 5]{GS} that this bound can be deduced from a result \cite[Theorem 2]{GW} on the so-called sumsets structure; thus we have a second combinatorial proof. For the L'vovsky's bound, the question remains open.
 
 Our method is purely combinatorial. It is based on the study of Ap\'ery sets and related invariants such as Frobenius numbers and degrees of elements in a numerical semigroup. So it provides a bridge between this problem and some classical aspects in Number Theory. Moreover, our approach leads to new characterization of the (arithmetically) Cohen--Macaulay and Buchsbaum property of $C$. This is the second aim of this paper. Note that this problem is a particular case of a more general problem of determining the Cohen-Macaulay property of the so-called simplicial affine semigroup rings. The latter problem was studied earlier by many authors, see e.g., \cite{GSW, St, TH}. The curve case was studied in \cite{ Br, T, TH, HeS}. 
 
 Using the Ap\'ery sets and degrees of its elements, we can describe the first local cohomology module $H^1_{\mfr}(K[S])$ explicitly, at least from the computational point of view, see Lemma \ref{A3} and Theorem \ref{A4a}. From that we can give   criteria for the Cohen--Macaulay and Buchsbaum property of $C$, see Theorem \ref{A3New}. Using these criteria, we can classify some new classes of monomial curves, see Proposition \ref{A3Add}, Corollary \ref{A3b} and Corollary \ref{BbmReg2}. In \cite{HeS}, the Cohen--Macaulay property of $C$ is studied via Gr\"obner bases. 
 
 Back to the Castelnuovo--Mumford regularity, this number is defined by two invariants: $a_1(K[S])$ and $a_2(K[S])$ (see the definition at the beginning of Section \ref{Breg}). Using Apery sets one can give a formula to compute $a_2(K[S])$ (see Lemma \ref{B1}) or to bound it in terms of Frobenius numbers (see Lemma \ref{B2}), or in terms of $a_1, a_{k-1}, d$ (see Lemma \ref{FrobSum}), which is smaller than L'vovsky's bound. The main task in bounding $\reg(C)$ is to bound $a_1(K[S])$.  Thank to the structure Theorem \ref{A4a} of $H^1_{\mfr}(K[S])$ we can give  lower and upper bounds on  $a_1(K[S])$, see Theorem \ref{A4}. Unfortunately, the obtained bounds are not explicit in terms of $a_1,\ldots,a_k$, and therefore we cannot reprove L'vovsky's bound. However, using this theorem, under some additional assumptions, we can provide  explicit bounds on $\reg(C)$, which are better than L'vovsky's bound, see Proposition \ref{Bd2} and Theorem \ref{A9}. It is interesting to mention that Theorem \ref{A9} is based on an improvement of a rather old bound on the Frobenius number given in \cite{Se}.

The structure Theorem \ref{A4a} together with  Lemma \ref{B1} allow us to  provide  algorithms to check the Cohen--Macaulay or Buchsbaum property of $C$,  to compute $\reg(C)$  and various other invariants, see the Appendix. Note that $a_1([K[S]), \ a_2([K[S]),\  \reg(K[S])$ can be  computed, using a special kind of initial ideals, see \cite[Example 4.3]{BGi} and the comment before it. Our algorithms do not require computation of Gr\"obner bases, and thus can run quickly.

The study of Castelnuovo--Mumford regularity has nice application to the sumsets structure theory. This is  a classical, but still hot topic in Additive Number Theory. Interested readers can consult recent papers \cite{El, GS} for some relationships between the two fields. In this paper we provide a new application, which states  that L'vovsky's bound also works for the so-called sumsets regularity $\sigma(A)$ (see Definition \ref{Sumreg} and Proposition \ref{SumsetLv}). It is worth to mention that the best bound on $\sigma(A)$ until this result is the Gruson--Lazarsfeld--Peskine bound  proved in \cite{GW} not long ago.

We now briefly describe the content of this paper. In Section \ref{Apery} we recall some basic facts and results on Ap\'ery sets and  Frobenius numbers. There we also provide a new bound on the Frobenius number (Lemma \ref{Selmer1}). In Section \ref{Locoh} we show how one can use Ap\'ery sets to describe the local cohomology modules of $K[S]$.  Using  Lemma \ref{A3} we can provide criteria for $K[S]$ to be a Cohen--Macaulay or Buchsbaum ring (Theorem \ref{A3New}), and give some explicit classes in terms of $a_1,\ldots,a_k$, see Proposition \ref{A3Add} and Corollary \ref{A3b}. In Section \ref{Breg}, we prove a fundamental result (Theorem \ref{A4a}) on the first local cohomology module $H^1_{\mfr}(K[S])$. Then we provide several bounds on $a_1(K[S])$ and $\reg(K[S])$ (see Theorem \ref{A4}, Proposition \ref{Bd2} and Theorem \ref{A9}). In the last Section \ref{Appl} we provide some examples and give the proof of Lvovsky's bound on the sumsets regularity of a numerical finite subset (Proposition \ref{SumsetLv}). The algorithms are given in the Appendix.

\vskip0.5cm
\noindent {\bf Acknowledgement}.  It is our pleasure to devote this paper to the memory of Professor J\"urgen Herzog. His works have inspired and still inspire a lot of researchers. In particular, his paper  \cite{He} made a great impact on  the study of (projective) monomial curves.

The first author of this work is partially supported by  the Program for Research Activities of Senior Researchers of VAST under the grant number  NVCC. 

\section{Frobenius number and Ap\'ery set} \label{Apery}

In this section we recall the notions of Frobenius number and Ap\'ery set. We also provide a new bound on Frobenius number which will be used in Section \ref{Breg} to bound the Castelnuovo--Mumford regularity.

We denote the set of non-negative integers and the set of integers by $\Nset$ and $\Zset$, respectively. If $a$ is a real number, we denote by $\lfloor a \rfloor$ the largest integer not exceeding $a$ and by $\lceil a \rceil$ the smallest integer $n$ such that $n\geq a$. The notation $[p,q]$ means the set $ \{n\in \Zset\mid p\leq n\leq q\}$, where $p\leq q$ are integers. If $A$ is a subset of an additive  semigroup $S$ and  $n$ is a positive integer, we denote by $nA$ the sum $A+\cdots + A \subseteq S$ ($n$ times).

Let $\Acal = \{ \alpha_1,  \ldots , \alpha_k \} $  be a set of integers $0<  \alpha_1 < \cdots < \alpha_k$ with $\gcd(\alpha_1,\ldots,\alpha_k) = 1$, where $k\ge 2$. In this paper we always assume that all coefficients $x_i$ in a sum $\sum_i x_i\alpha_i$ are non-negative integers.
We denote by $\langle \Acal \rangle := \{\sum_{i=1}^k x_i\alpha_i \mid x_i\in \Nset\}$ the numerical semigroup generated by $\Acal$. For each $0 \le i \le \alpha_k -1$, let
$$\omega_{\Acal}(i) = \min\{ \alpha \in \langle \Acal \rangle \mid  \alpha \equiv i \bmod \alpha_k \}.$$
For convenience, we also set $ \omega_{\Acal}(\alpha_k)   := 0$. The set 
$$\mathrm{Ap}_{\Acal} := \{ \omega_{\Acal}(0), \omega_{\Acal}(1),\ldots, \omega_{\Acal}(\alpha_k -1) \}$$  is called  {\it Ap\'ery set} of $\Acal$.
The greatest integer number not belonging to $\langle \Acal \rangle $ is called  {\it Frobenius number} and is denoted by $F( \alpha_1, \ldots ,\alpha_k)$ or $F(\Acal)$. Note that $F(\Acal) + 1$ is equal to the conductor of  the semigroup $\langle \Acal \rangle$. It is clear that
$$F(\Acal) = F( \alpha_1, \ldots ,\alpha_k) = \max\{\omega_{\Acal}(1),\ldots, \omega_{\Acal}(\alpha_k -1) \} - \alpha_k .$$
 More generally, let $\alpha\in \Acal$. If $t_0 = 0,\ldots,t_{\alpha -1} \in \langle \Acal \rangle$ such that $t_i \equiv i \bmod \alpha$, then 
 \begin{equation}\label{EAp0}
 F(\Acal) \le \max\{t_0, t_1,\ldots,t_{\alpha -1}\} - \alpha.
 \end{equation}
We often use the following formula
\begin{equation}\label{EAp1}
\omega_{\Acal}(i) - \alpha_k \le F( \alpha_1, \ldots ,\alpha_k).
\end{equation}
Note that for a fixed $i$ one can use Integer Programming to compute $\omega_{\Acal}(i)$: $\omega_{\Acal}(i) = i + n  \alpha_k$, where $n$ is the value of the following problem: 
$$\begin{array}{cc}
& n = \min  x_k       \\
{\mathrm{subject\  to} }     &  \begin{cases}
x_1\alpha_1 + \cdots + x_{k-1}\alpha_{k-1} = i + x_k \alpha_k,\\
x_1,\ldots,x_k  \in \Nset.
\end{cases}
\end{array}$$
However, the computation of Frobenius numbers is in general a NP-hard problem. Even in the case $k=3$, various polynomial time algorithms to compute $F(\Acal)$ are known but none lead to an explicit formula. The interested  readers can consult the book \cite{Ra} for more information and references. In the Appendix we give Algorithm 1 to compute the Ap\'ery set, from which we can also compute $F(\Acal)$.

For our application, we  need good upper bounds on $F(\Acal)$. A  simple bound is known as Schur's bound proved in \cite{Bra}
\begin{equation}\label{EFRb1}
F(\alpha_1,\ldots,\alpha_k) \le (\alpha_1-1)(\alpha_k - 1) - 1.
\end{equation}
The equality holds if $k=2$. We now recall  a bound obtained by Selmer. The numbers $\alpha_1,\ldots,\alpha_k$ are said to be {\it independent} if no $\alpha_i,\ 2\le i\le k$, can be written as  a sum $\sum_{j\le i-1} x_j\alpha_j$. Analyzing the proof of \cite[Theorem 1]{EG}, Selmer observed that if $\alpha_1,\ldots,\alpha_k$ are independent, then for each $i\le \alpha_1 - 1$, there is $t_i \in 
\langle \Acal \rangle$ such that $t_i \equiv i \bmod \alpha_1$, $t_i = \sum_{j\le k} x_j\alpha_j$ and $\sum_{j\le k} x_j \le 2\lfloor \alpha_1/k \rfloor$.  From this he got
$$F(\alpha_1,\ldots,\alpha_k) \le 2 \alpha_k \left\lfloor \frac{\alpha_1}{k } \right\rfloor - \alpha_1,$$
provided that $\alpha_1,\ldots,\alpha_k$ are  independent. In fact, this bound still holds under a weaker assumption as follows.

\begin{Lemma} \label{Selmer1} Assume that $\alpha_i \not\equiv \alpha_j \bmod \alpha_1$ for all $i\neq j$. Then
$$F(\alpha_1,\ldots,\alpha_k) \le 2 \alpha_k \left\lfloor \frac{\alpha_1}{k } \right\rfloor - \alpha_1.$$
\end{Lemma}

\begin{proof} For the convenience of the reader, we give  here a sketch of proof.  Let $m:= \left\lfloor \frac{\alpha_1}{k } \right\rfloor $ and $\overline{\Acal} := \{ 0, \overline{\alpha}_2,\ldots, \overline{\alpha}_k \}$ the set of residues of $\alpha_1,\ldots,\alpha_k$ modulo $\alpha_1$. We can consider $\overline{\Acal}$ as a subset of the group $\Zset/\alpha_1\Zset$.
By the assumption, the number of elements  $\sharp \overline{\Acal} = k$. Using this notation, it is shown in the proof of \cite[Theorem 1]{EG} that $m \overline{\Acal} + m \overline{\Acal}$ contains a complete residues system modulo $\alpha_1$. Then one can choose $t_i \in (2m) \Acal 
$ such that  $t_i \equiv i \bmod \alpha_1$. Since $t_i \le 2m \alpha_k$,  and by (\ref{EAp0}),  
 $F(\Acal) \le \max\{ t_1,\ldots,t_{\alpha_1-1}\} - \alpha_1$,  the statement follows.
\end{proof}

For an example, $\left\{ 5,7,9,16 \right\}$ satisfy the assumption of the above lemma, but they are dependent. We can reformulate the above result as the following:

\begin{Lemma} \label{Selmer2} Assume that the number of residues of $\alpha_1,\ldots, \alpha_k$ modulo $ \alpha_1$ is $k'$. Then
$$F(\alpha_1,\ldots,\alpha_k) \le 2 \alpha_k \left\lfloor \frac{\alpha_1}{k'} \right\rfloor - \alpha_1.$$
\end{Lemma}

\section{Local cohomology modules of a monomial curve} \label{Locoh}

We now see how one can use Ap\'ery sets to study the local cohomology modules of the coordinate ring of a projective monomial curve $C \subset \Pset^k_{K}$, which  is isomorphic to $K[s^d, s^{a_{k-1}}t^{d-a_{k-1}},\ldots, s^{a_1}t^{d-a_1}, t^d]$, where  $a_1< a_2 < \cdots < a_k =: d$ are  relatively prime positive integers. For convenience, we also set $a_0 = 0$. In order to explicitly indicate $C$, sometimes we also use the notation $C(a_1,\ldots,a_k)$. We only consider the case $k\ge 3$. Note that $C$ is smooth if and only if $a_1 = d-a_{k-1} = 1$.  We set  $\ebf_i : = (i,d-i)$ for all $i=0,\ldots,k$. Consider the following semigroups of $\Nset^2$:
$$\begin{array}{ll}
S &:=  \langle (a_i, d-a_i)\mid i= 0,\ldots,k\rangle = \langle \ebf_h\mid h = a_0, a_1,\ldots, a_k\rangle    \subset \Nset^2 , \\
S' & := \{ \ubf \in \Nset^2 \mid \ubf + p\ebf_0 \in S \ \text{and}\ \ubf + p\ebf_d \in S\ \text{for some }\ p \in \Nset\}.
\end{array}$$
We also use the following notations:
$$\begin{array}{ll}
A_1& := \{0, a_1,\ldots,a_k = d\}; \ S_{(1)} := \langle A_1\rangle,\\
A_2 & :=  \{0, d-a_{k-1},\ldots, d - a_1, d\}; \ S_{(2)} := \langle A_2\rangle, \\
 \mathrm{Ap}_i & := \mathrm{Ap}_{A_i\setminus \{0\}} := \{\omega_i(0),\ldots,\omega_i(d-1)\} \ (i=1,2), \\
 H & := \Zset S = \{ \ubf = (u_1,u_2) \in \Zset^2\mid u_1 + u_2 \equiv 0 \bmod d\}.
\end{array}$$
Note that $S\subseteq S' \subseteq H$ and by \cite[Lemma 2.1]{Br},  $S'= (S_{(1)} \times S_{(2)}) \cap H$. We equip rings $K[S] \subset K[H]$ with $\Zset$-grading by setting
$$\deg (\ubf) : = \frac{u_1 +u_2}{d},$$
for each $\ubf \in H$.
 Let $\mfr := K[S\setminus \left\{ 0 \right\}]$ be the maximal homogeneous ideal of $K[S]$. Then the local cohomology modules of $K[S]$ can be described as follows.

\begin{Lemma} \label{A1}
\begin{enumerate}[\((i)\)]
    \item $S' = [ H \cap (\mathrm{Ap}_1 \times \mathrm{Ap}_2)] + \Nset\ebf_0 + \Nset \ebf_d$,

As graded modules, we have
\item $H^1_{\mfr} (K[S]) \cong K[S' \setminus S]$,
\item $H^2_{\mfr} (K[S]) \cong K[H \cap( (\Zset\setminus S_{(1)}) \times (\Zset\setminus S_{(2)})) ]$.
\end{enumerate}
\end{Lemma}

\begin{proof}
\begin{enumerate}[\((i)\)]
    \item is \cite[Lemma 2.4]{FH}.
    \item is  \cite[Lemma 2.6]{Br}. It is also a particular  case of \cite[Corollary 3.4]{TH}.
    \item is \cite[Lemma 2.6\((i)\)]{HeHS}.
\end{enumerate}
\end{proof}
Let
\begin{equation}\label{wec}
\wbf_j = (\omega_1(j), \omega_2(d-j)) \ \ (0\le j\le d-1).
\end{equation}
With this notation, we can rewrite Lemma \ref{A1}\((i)\) as
\begin{equation}\label{Ewec2}
S' = \{ \wbf_0,\ldots, \wbf_d\} + \Nset \ebf_0 + \Nset\ebf_d
\end{equation}
From the definition of the Ap\'ery set, we can see that
$$\Zset \setminus \langle S_{(i )}\rangle = \bigsqcup_{0\le j\le d-1}
 \{\omega_i(j) - nd\mid n> 0 \}\enskip\text{(disjoint union).}$$
By  Lemma \ref{A1}\((iii)\), as graded  $K$-vector spaces, we  can write
\begin{equation}\label{Eloco2} H^2_{\mfr} (K[S]) \cong  \bigoplus_{j=0}^{d-1}\bigoplus_{n,m>0} K\cdot(\wbf_j - n\ebf_0 - m\ebf_d).
\end{equation}
It is much more difficult to handle with $H^1_{\mfr} (K[S]) $, because we do not know when does an element $\ubf \in S'$ not belong to $S$. The technique used in \cite{Br, FH} is  useful in this paper. 

Let  $\Acal =\{\alpha_1< \alpha_2<\cdots <  \alpha_k\}$ and $n\in \langle \Acal \rangle$. The degree of $n$ with respect to 
$\Acal$ is the number
$$\delta_{\Acal} (n) = \min\{ x_1+ \cdots + x_k\mid n = x_1\alpha_1 + \cdots + x_k\alpha_k,\enskip x_1,\ldots,x_k \in \Nset\}.$$
Note that $\delta_{\Acal}(0) = 0$, and the value of $\delta_{\Acal}(n)$ depends on the choice of  a (not necessary minimal)  generating set $\Acal$ of the semigroup $ \langle \Acal \rangle$. If $n\not\in \langle \Acal \rangle$, we set $\delta_{\Acal}(n) := \infty$.

\begin{Remark}\label{degree} 
\begin{enumerate}[\(1)\)]
    \item As shown in \cite{Br, FH}, it is important to calculate or estimate  $\delta_{\Acal} (\omega_{\Acal}(i)) $. It is clear that if $n\in \langle \Acal \rangle$ and $n\leq b$, then $\delta(n) \leq \lfloor b/\alpha_1\rfloor $. However, in many cases this estimation is too big.

For any number $n,m \in \Nset$, we have
$$\delta_{\Acal} (m+n)  \le \delta_{\Acal} (m) + \delta_{\Acal} (n).$$
However, it could happen that $ \delta_{\Acal} (m) \gg \delta_{\Acal} (m+n) $. In particular, one could have  $\delta_{\Acal} (n) \gg \delta_{\Acal} (n+d)$. This causes estimation of $\delta_{\Acal} (\omega_{\Acal}(i)) $ complicated even if we know a bound from $\langle \Acal\rangle$ on $\omega_{\Acal}(i)$ and the degree of that  bound.

For an example, let $\Acal =\{1, d-1, d\}$ ($d\geq 4$). Then $\delta_\Acal(2d-2) = 2$. Since $1\in \Acal$,  $\omega_{\Acal}(d-2) = d-2$, whence $\delta_{\Acal} (\omega_\Acal(d-2)) = d-2$. This number is  much larger than $\delta_\Acal(2d-2) = 2$, although $2d - 2 = \omega_{\Acal}(d-2) +d$ is a bound of $\omega_{\Acal}(d-2)$.
\item In order to calculate the degree of a given number, one can use Integer Programming: if $n\in \langle \Acal \rangle$, then 
$$\begin{array}{cc}
& \delta_{\Acal}(n) =  \min x_1 + \cdots +  x_k       \\
\text{ subject  to}      &
\begin{cases}
x_1\alpha_1 + \cdots + x_{k}\alpha_{k} = n,\\
x_1,\ldots,x_k  \in \Nset.
\end{cases}
\end{array}$$ 
However, for calculating degrees of a (larger) set of numbers, by \cite[Remark 2]{FH},  one can use the following recursive formula: let
$$\begin{array}{ll}
\delta_{\Acal}(0 ) & := 0, \\
\delta_{\Acal}(1 ) & :=  \begin{cases} 1 & \text{if} \ 1\in \Acal, \\
\infty  & \text{if} \  1\not\in \Acal, 
\end{cases}\\
\delta_{\Acal} (n)&  = 1 +  \min\{ \delta_{\Acal}(n-\alpha_i) \mid \alpha_i \le n;\  i=1,2,\ldots,k\} \ \text{for} \ n\geq 2.
\end{array}$$ 
\item When $\Acal$ is known, and no confusion is arisen, we simply write $\omega(i)$ and $\delta(n)$ instead of $\omega_\Acal(i)$ and $\delta_\Acal(n)$.
\end{enumerate}
\end{Remark}
For short, let $\delta_i(n) := \delta_{A_i \setminus \{0\}}(n)$, provided $n\in S_{(i)}$ ($i=1,2$).  If $\ubf = (u_1,u_2) \in S_{(1)} \times S_{(2)}$, for short, we set $\delta_i (\ubf) := \delta_i(u_i)$ ($i=1,2$). 

 The following easy result is a  useful tool to compute $S'\setminus S$.

\begin{Lemma} \label{A2} {\rm (\cite[Lemma 3.1]{Br}, \cite[Lemma 1]{Ka})}. Let $\ubf \in S'$. Then the following are equivalent:
\begin{enumerate}[\((i)\)]
    \item $\ubf \in S$,
    \item $\delta_1(\ubf) \leq \deg(\ubf)$,
    \item $\delta_2(\ubf) \leq \deg(\ubf)$.
\end{enumerate}
\end{Lemma}

Some criteria for the Cohen--Macaulay property of $K[S]$ were given earlier. The most used is the one given in \cite[Theorem 5.1]{GSW} and \cite[Theorem 6.4]{St}, where a criterion was given for a  more general class of ring, also see \cite[Corollary 4.4]{TH}. For the case of monomial curves, there are specific criteria given in \cite[Lemma 3.1]{T}, \cite[Proposition 2.5]{FH} and \cite[Theorems 2.2, 3.2, 3.6]{HeS}. We will see that one can use Ap\'ery sets to check if $K[S]$ is a Cohen--Macaulay ring or not.  For that we need an auxiliary result.

\begin{Lemma} \label{A3} Let
$$\Iset := \{ 1\le i\le d-1\mid i\not\in \{a_1,\ldots,a_{k-1}\},\ \delta_1(\wbf_i ) > \deg(\wbf_i)\}.$$
Then 
$$\Iset := \{ 1\le i\le d-1\mid i\not\in \{a_1,\ldots,a_{k-1}\},\ \delta_2(\wbf_i) > \deg(\wbf_i)\},$$
and
$$  \{\wbf_i\mid i \in \Iset\} \subseteq S'\setminus S \subseteq  \{\wbf_i\mid i \in \Iset\} + \Nset\ebf_0 + \Nset \ebf_d.$$
\end{Lemma}

\begin{proof} 
By (\ref{Ewec2}), $\wbf_i\in S'$ for all $i$. By Lemma \ref{A2}\((ii)\) $\Leftrightarrow (iii)$, we get the first equality. Using the equivalence $(i)\Leftrightarrow (ii)$ of Lemma \ref{A2}, we have 
$ \wbf_i \not\in S$ if and only if $i \in \Iset$, which implies the inclusion $ \{\wbf_i\mid i \in \Iset\} \subseteq S'\setminus S$. The last inclusion now follows from (\ref{Ewec2}).
\end{proof}

Recall that $K[S]$ is called a Buchsbaum ring if $\mfr H^1_\mfr(K[S]) = 0$. With the above lemma we can give a quite effective characterization for $K[S]$ to be  a Cohen--Macaulay or Buchsbaum ring. Based on this characterization we can construct algorithms to verify these properties, see Algorithms 2 and 3 in the Appendix. 

\begin{Theorem} \label{A3New} 
\begin{enumerate}[\((i)\)]
    \item The ring $K[S]$ is  Cohen--Macaulay if and only if $\Iset = \emptyset$.
    \item Assume that $\Iset \neq \emptyset$. 
If the condition 
\begin{equation} \label{EA3N}
\delta_1(\wbf_i) = \delta_2(\wbf_i) = \deg(\wbf_i) + 1
\end{equation}
 is satisfied for all $i\in \Iset$, then $H^1_\mfr(K[S])\cong \bigoplus_{i\in \Iset}K\wbf_i$ as $K$-vector spaces. In particular,   $\ell ( H^1_\mfr(K[S])) = \sharp(\Iset) \leq d- k$.
 \item The ring $K[S]$ is  non-Cohen--Macaulay Buchsbaum if and only if $\Iset \neq \emptyset$, the condition (\ref{EA3N}) is satisfied  for all $i \in \Iset$ and there is no pair $\{i,j\} \subseteq \Iset$ such that  $\wbf_j - \wbf_i = \ebf_h$ for some $h\in \{a_1,\ldots,a_{k-1} \}$.
\end{enumerate}
\end{Theorem}

\begin{proof} 
\begin{enumerate}[\((i)\)]
    \item By Lemma \ref{A1}\((ii)\), $K[S]$ is a Cohen-Macaulay ring if and only if $S'\setminus S= \emptyset$. By the  inclusions in Lemma \ref{A3},  this is equivalent to $\Iset = \emptyset$.
    \item Under the assumption (\ref{EA3N}), by Lemma \ref{A2}, we can immediately see that $\wbf_i + \ebf_0\in S$ and $\wbf_i + \ebf_d \in S$ for all $i$. Hence, by Lemma \ref{A3}, $S' \setminus S= \{\wbf_i\mid i \in \Iset\}$. By Lemma \ref{A1}\((ii)\), this implies $H^1_\mfr(K[S])\cong \bigoplus_{i\in \Iset}K\wbf_i$ as $K$-vector spaces, whence   $\ell ( H^1_\mfr(K[S])) = \sharp(\Iset) \leq d- k$.
    \item \textit{Necessity.} Assume that $K[S]$ is a non Cohen--Macaulay Buchsbaum ring. By \((i)\), $\Iset \neq \emptyset$. Since $\mfr H^1_\mfr(K[S]) = 0$ and $ \{\wbf_i\mid i \in \Iset\} \subseteq S'\setminus S$, by Lemma \ref{A1}\((ii)\), $\wbf_i + \ebf_0 \in S$ and $\wbf_i + \ebf_d \in S$.  By the inclusions in Lemma \ref{A3} we must have
$S'\setminus S =  \{\wbf_i\mid i \in \Iset\}$. 

Fix $i \in \Iset$. By the definition of $\Iset$, $\delta_1(\wbf_i) \geq \deg(\wbf_i) + 1$. Since $\wbf_i + \ebf_0 \in S$, by Lemma \ref{A2}, $ 1+ \deg(\wbf_i) =  \deg(\wbf_i + \ebf_0) \geq \delta_1( \wbf_i + \ebf_0) = \delta_1( \wbf_i)$. Hence $\delta_1(\wbf_i) =  \deg(\wbf_i) + 1$.

Using the symmetry given in Lemma \ref{A3} in the description of $\Iset$, we also get $\delta_2(\wbf_i) = \deg(\wbf_i) + 1$.  Thus the condition (\ref{EA3N}) holds.

Finally, if there are $i\neq j$ from $\Iset$ such that $\wbf_i-\wbf_j =\ebf_h$ for some $h \in \{a_1,\ldots,a_{k-1} \}$, then $\ebf_h + \wbf_j = \wbf_i \not\in S$, whence $\mfr H^1_{\mfr}(K[S]) \neq 0$, a contradiction. So, the last condition in \((iii)\) holds.

\textit{Sufficiency.} Since $\mathbb{I}\neq\emptyset$, $K[S]$ is not Cohen--Macaulay ring by \((i)\). Under the condition (\ref{EA3N}),  as observed in \((ii)\), $\wbf_i + \ebf_0\in S$ and $\wbf_i + \ebf_d \in S$ and $S'\setminus S = \{\wbf_i\mid i \in \Iset\}$. Hence,  if $\ebf_h + \wbf_i \in S$ for all $h\in A_1$ and $i\in \Iset$, then we $\mfr H^1_{\mfr}(K[S]) = 0$. We now assume that $K[S]$ is not a Buchsbaum ring. Then there must be $i\in \Iset$ and $\ebf_h,\ h\in A_1$ such that $\wbf_i + \ebf_h \not\in S$. Since $S'\setminus S = \{\wbf_i\mid i \in \Iset\}$ and $\wbf_i + \ebf_h \in S'$,  this implies $\wbf_i + \ebf_h =\wbf_j$ for some $j\in \Iset$. In particular $i\neq j$. By the definition of Ap\'ery set, we must have $h\neq 0, d$. So, $\wbf_j - \wbf_i = \ebf_h$ for some  $h\in \{ a_1,\ldots,a_{k-1}\}$, a contradiction to the assumption. Hence $K[S]$ is  a Buchsbaum ring.
\end{enumerate}
\end{proof}

Note that there is a criterion given in \cite[Theorem 3.1]{Br}  for $K[S]$ to be a Buchsbaum ring in terms of degrees of elements of Ap\'ery set and their components. The criterion also contains (\ref{EA3N}). One can deduce it from the statement \((iii)\) of above theorem. However, we believe that \((iii)\) is easier to check.

\begin{Example} \label{A3c1} {\rm Let $C_1 = C(1,3,4,8,10)$. Then $\wbf_i = (i, 20-i)$ for $i = 5,7,9$ and $\wbf_i = (i,10-i)$ in other cases. Hence $\Iset= \{2, 6 \}$. We have  $\delta_1(\wbf_2) = \delta_2(\wbf_2)  = 2 = \deg(\wbf_2) +1$ and  $\delta_1(\wbf_6) = \delta_2(\wbf_6)  = 2 = \deg(\wbf_6) +1$. That means  (\ref{EA3N}) is satisfied. The last condition is void in this case. Hence $C_1$ is a Buchsbaum curve.

Now let $C_2 = C(1,3,4,10,12)$. Then $\wbf_i = (i, 24-i)$ for $i = 5,7,9,11$ and $\wbf_i = (i,12-i)$ in other cases. Hence 
$\Iset= \{2,6, 8, 9\}$. One can check that the condition  (\ref{EA3N}) is satisfied for $\delta_1(\omega_1(i)),\ i\in \Iset$. Since $\delta_2(\omega_2(6)) =  3 > \deg(\wbf_6) + 1=2$, the condition  (\ref{EA3N}) is not satisfied for $\delta_2(\omega_2(6))$. Therefore $C_2$ is not a Buchsbaum curve.

Finally, let $C_3= C(1,4,21,85)$. Applying Algorithm 2 in the appendix, we get  $\Iset =\{20, 41, 62, 83\}$. We have $ \wbf_{20} = (20, 320),\ \wbf_{41}= (41,384),\ \wbf_{62} = (62, 448),\ \wbf_{83} = (83,512)\}$. All these four elements satisfy the condition (\ref{EA3N}). However,  $\wbf_{41} - \wbf_{20} = (21,64) = \ebf_{21}$. So, $C_3$ is not a Buchsbaum curve. 

For more curves, see Example \ref{ExAl2}. }
\end{Example}

In order to apply  Theorem \ref{A3New}, one should calculate the set $\Iset$ and degrees of it elements. For an arbitrary setting, it is  probably impossible to have explicit criterion in terms of $a_1,\ldots,a_k$. However, in some special cases, there are such criteria. Following \cite{T, LT} we can rewrite $\{a_0,\ldots,a_k\} = \bigsqcup_{j=0}^r [b_{2j},b_{2j+1}]$, where $0=:b_0 \leq b_1 < b_2 \leq b_3 < \cdots < b_{2r}\leq b_{2r+1} := d$, and $b_{2j-1} + 2 \leq b_{2j}$ for $j=1,\ldots,r$. In this new presentation, $0 < b_1$ is equivalent to $a_1 = 1$. If $b_1 > 0$ and $b_{2r} < d$ (i.e. $C$ is a smooth curve), then $K[S]$ is Cohen--Macaulay ring if and only if $A_1 = [0,d]$. If $C$ is not smooth (or equivalently, we may assume that $a_1 > 1$), then in \cite{T}, the main results concerning the Cohen--Macaulay property of $K[S]$ are only given  for  the case $r=2$ (see Theorem 2.1 and Theorem 3.5 there). Using Theorem \ref{A3New}\((i)\) we can derive some other results for the case $r> 2$.

\begin{Proposition} \label{A3Add} Let $\epsilon = \max\{a_1, d-a_{k-1}\}$.  Assume that  $i, i+1,\ldots, i+\epsilon -1  \in A_1$ for some $0\leq i \leq d-\epsilon + 1$. Then $K[S]$ is a Cohen--Macaulay ring if and only if the following conditions are satisfied:
\begin{enumerate}[\((i)\)]
    \item For all $x \in [0,i-1]\setminus A_1$ we have $x\not\in \langle A_1\rangle$ and $x+d \in 2A_1$.
    \item For all $x \in [i+\epsilon]\setminus A_1$ we have  $d-x\not\in \langle A_2\rangle$ and $x \in 2A_1$.
\end{enumerate}
\end{Proposition}

\begin{proof} Under the assumption, by (\ref{EAp0}), we have $F_1 \leq i-1$ and $F_2 \leq d- (i+\epsilon)$. Hence, for $x\not\in A_1$,
$$\wbf_x = \begin{cases}
(x, d-x) & \text{if}\quad x\leq i-1 \ \text{and}\ x\in \langle A_1\rangle,\\
(x+d, d-x)  & \text{if}\quad x\leq i-1 \ \text{and}\ x \not\in \langle A_1\rangle,\\
(x, d-x)  & \text{if}\quad x\geq i+\epsilon \ \text{and}\ d-x\in \langle A_2\rangle,\\
(x, 2d-x)  & \text{if}\quad x\geq i+\epsilon \ \text{and}\ d-x \not\in \langle A_2\rangle.
\end{cases}$$
Assume that $K[S]$ is a Cohen--Macaulay ring.
 If $x\leq i-1$ and $ x\in \langle A_1\rangle$, then $\delta_1(\wbf_x) \geq 2$, while $\deg(\wbf_x) = 1$, so $x\in \Iset$. Similarly, if $x\geq i+\epsilon$ and $ d- x\in \langle A_2\rangle$, then $x\in \Iset$. By  Theorem \ref{A3New}\((i)\) these two cases cannot happen.
 
  In the other two cases, again by Lemma \ref{A3},   $\Iset = \emptyset$ implies that $\delta_1 (x+d) \leq 2$ and $\delta_1(x)\leq 2$, respectively. Since $x+d,\ x\not\in A_1$, we must have  $\delta_1 (x+d) = 2$ and $\delta_1(x)=  2$. In other words, \((i)\) and \((ii)\) hold.
  
  The converse holds by using Lemma \ref{A3} and Theorem \ref{A3New}\((i)\).
\end{proof}

We call $\lambda_i := a_i - a_{i-1} - 1$ the $i$-th gap of $A_1 = \{0,  a_1,\ldots,a_k\}$, where $1\le i\le k$  and $a_0 := 0$. If $\lambda_i > 0$, it is called a {\it  non-zero gap} (which is real meaning of a gap!). Below is another partial result.

\begin{Corollary} \label{A3b} Let  $\{a_0,\ldots,a_k\} = \bigsqcup_{j=0}^r [b_{2j},b_{2j+1}]$, where $r\geq 1$,  $0=:b_0 \leq b_1 < b_2 \leq b_3 < \cdots < b_{2r}\leq b_{2r+1} := d$, and $b_{2j-1} + 2 \leq b_{2j}$ for $j=1,\ldots,r$.  Assume that $a_1=1$.  If there is $1\leq j\leq r$ such that  $b_{2j+1} - b_{2j-1}\geq \lambda_k + 2$ ($\lambda_k$ is the last gap, which could be zero), then $K[S]$ is not a Cohen--Macaulay ring.
\end{Corollary}

\begin{proof} Assume that $b_{2j-1} = a_{i-1}$. Then  $a_i = b_{2j}$. First consider the case $\lambda_i > \lambda_k$. We have $\lambda_i = a_i - a_{i-1} - 1 > 0$. Since $a_1 =1$, $1 < i < k$. Let $x := a_i + a_{k-1} - d$. Then  $x< a_i$ and $x= a_i - 1 - \lambda_k = a_{i-1} + \lambda_i - \lambda_k > a_{i-1}$. This means that $x\not\in A_1$. Since $a_1 = 1 \in A_1$, $\omega_1(i) =i$ for all $i< d$. Hence, $\delta_1(\wbf_x) = \delta_1(x) \geq 2$. On the other hand, $ d-x = (d-a_i) + (d-a_{k-1}) \in \langle A_2 \rangle$, whence $\wbf_x = (x, d-x)$ and $\deg(\wbf_x) = 1 <   \delta_1(\wbf_x) $. By Lemma \ref{A3}, $x\in \Iset$. Hence, by Theorem \ref{A3New}\((i)\), $K[S]$ is not a Cohen--Macaulay ring.

We now consider the case $\lambda_i \leq \lambda_k$. Then $b_{2j+1} \geq b_{2j} + 1$ (otherwise $\lambda_k+2 \leq b_{2j+1} - b_{2j-1} = \lambda_i +1$, a contradiction). We have $a_{i-1} + 1 +  d= y + a_{k-1}$, where 
\begin{align*}
    y & = a_{i-1} +1 +  d- a_{k-1} \\
& = a_{i-1} + \lambda_k +2  = \lambda_k - (a_i - a_{i-1} -1) + a_i +1  = \lambda_k - \lambda_i + a_i  +1 > a_i  = b_{2j}.
\end{align*}
 By the assumption,  $b_{2j+1} \geq b_{2j-1} + \lambda_k + 2 = a_{i-1} + \lambda_k + 2 = y$, which implies  $y \in [b_{2j}, b_{2j+1}] \subset A_1$. Therefore $d- (a_{i-1} + 1) = (d-y) + (d- a_{k-1}) \in \langle A_2 \rangle$, whence $\omega_2(d-x) = d-x$ and $\wbf_x = (x, d-x)$, where $x= a_{i-1} + 1$.  Since $\delta_1(\wbf_x) = 2 > \deg(\wbf_x) = 1$, $x\in \Iset$ by Lemma \ref{A3}. Hence, by Theorem \ref{A3New}\((i)\),  $K[S]$ is not Cohen--Macaulay.
\end{proof}

\section{Bounding Castelnuovo--Mumford regularity} \label{Breg}

The purpose of this section is to give bounds on the Castelnuovo--Mumford regularity of $K[S]$. If $R$ is a $d$-dimensional standard graded $K$-algebra with the maximal irrelevant ideal $\mfr$, then 
for $ i \leq d$ we set
$$ a_i(R) := \sup\{ n\mid  H^i_{\mfr} (R)_n \neq 0 \}.$$
The {\it Castelnuovo--Mumford regularity} of  $R$ is the number
$$\reg(R) = \max \{ a_i(R) + i\mid 0\leq i \leq d\}.$$
Since $K[S]$ is a domain, the  Castelnuovo--Mumford regularity of $K[S]$ is the number
$$\reg(K[S]) = \max\{ a_1(K[S]) + 1;\enskip a_2(K[S]) + 2\}.$$
Note that the Castelnuovo--Mumford regularity of the corresponding monomial curve $C$ is the number $\reg(C) := \reg(K[S])+1$.

If $A_1 = [0,d]$, it is well-known that $K[S]$ is a so-called normal ring, hence a Cohen--Macaulay. The latter property also immediately follows from  Lemma \ref{A1}\((ii)\).  Moreover, from  Lemma \ref{A1}\((iii)\), $H^2_{\mfr} (K[S]) $ has no  non-negative grading, whence $\reg(K[S]) = 1$. Hence, {\it from now on we always assume that $k< d$}.

Using (\ref{Eloco2}) we can compute $a_2(K[S])$.

\begin{Lemma} \label{B1} 
$$a_2(K[S]) =  \max_{0 \le i \le d-1}\{\deg(\wbf_i) \} - 2 =  \max_{0 \le i \le d-1}\frac{\omega_1(i) + \omega_2(d-i)}{d} - 2.$$
\end{Lemma}

Denote the Frobenius number of $A_i \setminus \{0\}$ by $F_i$, that is  $F_i := F(A_i \setminus \{0\})$ ($i =1,2)$. By (\ref{EAp1}), $\omega_i( j) \le F_i + d$ for all $j\le d-1$. Hence, a direct application of the above lemma gives  the following simple upper bound:

\begin{Lemma} \label{B2}  $a_2(K[S]) \le \left\lfloor \frac{F_1 + F_2}{d} \right\rfloor$.
\end{Lemma}

For explicit bounds in terms of $a_i$, we need the following result:
\begin{Lemma} \label{FrobSum} Let $\epsilon = \max\{a_1, d-a_{k-1}\} $.
\[\left\lfloor \frac{F_1+F_2}{d} \right\rfloor \le\left\{ \begin{array}{cl}
a_1 + (d-a_{k-1}) -3,\\
\epsilon-2 &\text{if}\enskip i, i+1 \in A_1\  \text{for some}\ 0\leq i \leq d-1,\\
0&\text{if}\enskip i, i+1,\ldots, i+\epsilon -1  \in A_1 \  \text{for some}\ 0\leq i \leq d-\epsilon + 1.
\end{array} \right.\]
\end{Lemma}

\begin{proof}
\begin{enumerate}[\((i)\)]
    \item Using Schur's bound (\ref{EFRb1}), we get
$$\left\lfloor \frac{F_1+F_2}{d} \right\rfloor \leq \left\lfloor \frac{(a_1-1) (d-1) + (d-a_{k-1}-1)(d-1) - 2 }{d} \right\rfloor  = a_1 + (d-a_{k-1}) - 3.$$
\item Note that $F_1 \leq F(a_1, i, i+1)$ and $F_2 \leq F(d-a_{k-1}, d- (i+1), d-i)$. So, applying Schur's bound to $F(a_1, i, i+1)$ and $ F(d-a_{k-1}, d- (i+1), d-i)$, we get
$$\left\lfloor \frac{F_1+F_2}{d} \right\rfloor \leq \left\lfloor \frac{(a_1-1)i + (d-a_{k-1} - 1)(d-i-1) - 2 }{d} \right\rfloor \leq \epsilon - 2.$$
\item Under the assumption, by (\ref{EAp0}) we see that $F_1 \leq i-1$ and $F_2 \leq d-( i+\epsilon )$. Hence, by Lemma \ref{B2} and the above discussion, we get
 $$\left\lfloor \frac{F_1+F_2}{d} \right\rfloor  \leq \left\lfloor \frac{i + d-(i+\epsilon) - 1}{d} \right\rfloor  = 0.$$
\end{enumerate}
\end{proof}
Thanks to the above lemmas, it is clear that the main task in  bounding $\reg(K[S])$  is to bound $a_1(K[S])$. 
We can below give a way to compute $a_1(K[S])$ in terms of Ap\'ery sets and degrees of certain numbers.

\begin{Theorem} \label{A4a}  Assume that $K[S]$ is a non-Cohen--Macaulay ring. Then  for each $i \in \Iset$, $\delta_2(\wbf_i) - \deg(\wbf_i) - 1\geq 0$, and for each $0\leq n_1 \le \delta_2(\wbf_i) - \deg(\wbf_i) - 1$, we have $ \delta_1(\omega_1(i) + n_1d) - \deg(\wbf_i) - n_1 -1 \geq 0$.
Let
$$\begin{array}{ll}
\Lcal :=  \{\wbf_i + (n_1d, n_2d)\mid i\in \Iset;\ & 0\leq n_1 \le \delta_2(\wbf_i) - \deg(\wbf_i) - 1,\\
& n_2 =  \delta_1(\omega_1(i) + n_1d) - \deg(\wbf_i) - n_1 -1\},\\
\Lcal' :=  \{\wbf_i + (n_1d, n_2d)\mid i\in \Iset;\ & 0\leq n_1 \le \delta_2(\wbf_i) - \deg(\wbf_i) - 1,\\
& 0\le n_2 \le  \delta_1(\omega_1(i) + n_1d) - \deg(\wbf_i) - n_1 -1\}.
\end{array}$$
Then $H^1_{\mfr}(K[S]) \cong \bigoplus_{\wbf\in \Lcal'} K\wbf$ as graded $K$-vector spaces. In particular, we have\newline $a_1(K[S]) = \max_{\ubf \in \Lcal} \deg(\ubf)$ and $\ell(H^1_{\mfr}(K[S]) = \sharp (\Lcal')$.
\end{Theorem}

\begin{proof} By Lemma \ref{A3}, $\wbf_i \in S'\setminus S$. Hence, by Lemma \ref{A2}\((iii)\) we have $\delta_2(\wbf_i) - \deg(\wbf_i) - 1\geq 0$. 

Let $0\leq n_1 \le \delta_2(\wbf_i) - \deg(\wbf_i) - 1$. Then $\delta_2(\wbf_i + n_1\ebf_d) =  \delta_2(\wbf_i) \geq n_1 +  \deg(\wbf_i) + 1 = \deg(\wbf_i + n_1\ebf_d) + 1$. By Lemma \ref{A2}\((iii)\), $\wbf_i + n_1\ebf_d \not\in S$. Since $\delta_1(\omega_1(i) + n_1d) = \delta_1(\wbf_i + n_1\ebf_d)$, by Lemma \ref{A2}\((ii)\), $\delta_1(\omega_1(i) + n_1d) - \deg(\wbf_i) - n_1 -1 \geq 0$.

Since $K[S]$ is a non-Cohen--Macaulay ring, by Theorem \ref{A3New}\((i)\),  $\Iset \neq \emptyset$. 
By Lemma \ref{A1}\((ii)\), $H^1_{\mfr}(K[S]) \cong K[S'\setminus S]$. Hence $a_1(K[S]) = \max_{\ubf \in S'\setminus S} \deg(\ubf)$ and $ \ell(H^1_{\mfr}(K[S]) = \sharp (S'\setminus S)$. Since $\max_{u\in \Lcal} \deg(\ubf) = \max_{u\in \Lcal'} \deg(\ubf)$, it  is enough to show that $\Lcal' = S'\setminus S$.

Let $\ubf \in S'\setminus S$. By Lemma \ref{A3}, $\ubf = \wbf_i + (n_1d, n_2d)$ for some $i\in \Iset$ and $n_1,n_2\in \Nset$. Since $\ubf \not\in S$, by Lemma \ref{A2}\((ii)\),
$$n_1 + n_2 + \deg(\wbf_i) = \deg(\ubf) \leq \delta_1(\ubf)  - 1 = \delta_1(\omega_1(i) + n_1d) -1,$$
whence
\begin{equation}\label{EA4a1}
n_2 \le \delta_1(\omega_1(i) + n_1d) - \deg(\wbf_i) - n_1 - 1.
\end{equation}
Similarly, by Lemma \ref{A2}\((iii)\), we have
\begin{align*}
    n_1 + n_2 + \deg(\wbf_i) = \deg(\ubf) & \leq \delta_2(\ubf)  - 1 = \delta_2(\omega_2(d-i) + n_2d) - 1\\
&\leq \delta_2(\omega_2(d-i)) + n_2 - 1.
\end{align*}
Hence 
\begin{equation}\label{EA4a2}
n_1 \leq \delta_2(\omega_2(d-i))  - \deg(\wbf_i) - 1 = \delta_2(\wbf_i) - \deg(\wbf_i) - 1.
\end{equation}
Combining (\ref{EA4a1}) and (\ref{EA4a2}) we get $\ubf \in \Lcal'$.

Conversely, let $\ubf = \wbf_i + (n_1d, n_2d) \in \Lcal'$. By (\ref{Ewec2}), $\ubf \in S'$. By definition of $\Lcal'$,
$$\delta_1(\ubf) = \delta_1( \omega_1(i) + n_1d) \geq \deg(\wbf_i) + n_1 + n_2 + 1 = \deg(\ubf) + 1.$$
By Lemma \ref{A2}\((ii)\), $\ubf \not\in S$. Hence $\ubf\in S' \setminus S$, as required.
\end{proof}

\begin{Remark} \label{A4R} {\rm We always have
 $\{\wbf_i|\ i\in \Iset\} \subseteq \Lcal'$.}
\end{Remark}

In the above theorem, besides Ap\'ery sets, we have to calculate the degrees of certain sums $\omega_1(i) + n_1d$. This is not an easy  task. As mentioned in Remark \ref{degree}, the larger $n_1$ is, the more difficult to compute $\delta_1(\omega_1(i)+ n_1d)$.  In the following lower and upper bounds,  only two Ap\'ery sets are involved. Therefore, these bounds are better to be applied.

\begin{Theorem} \label{A4} Assume that $\Iset \neq \emptyset$. Then
\begin{enumerate}[\((i)\)]
    \item $a_1(K[S])  \geq \max_{i\in \Iset}
\{\delta_1(\wbf_i),\ \delta_2(\wbf_i)\} - 1$;
\item $a_1(K[S])  \leq \max_{i\in \Iset}
 \{\delta_1(\wbf_i) +  \delta_2(\wbf_i) - \deg(\wbf_i) \} - 2 $;
 \item $a_1(K[S])  \leq  \left\lfloor \left( 1- \frac{a_1}{d} \right)  \max_{i\in \Iset} \{\delta_1(\wbf_i) \} +  \frac{a_{k-1}}{d}  \max_{i\in \Iset}
\{  \delta_2(\wbf_i) \} \right\rfloor - 2$.
\end{enumerate}
\end{Theorem}

\begin{proof} 
\begin{enumerate}[\((i)\)]
    \item Let $j\in \Iset$ such that
$$\max
\{\delta_1(\wbf_j),\ \delta_2(\wbf_j)\} = \max_{i\in \Iset}
\{\delta_1(\wbf_i),\ \delta_2(\wbf_i)\}.$$
W.l.o.g. we may assume that $\delta_1(\wbf_j)\geq  \delta_2(\wbf_j)$. Let $\wbf := \wbf_j + (0, n_2d)$, where $n_2 :=  \delta_1(\wbf_j) - \deg(\wbf_j) - 1$. Then $\wbf \in \Lcal$ (with $n_1 = 0$). Hence, by Theorem \ref{A4a}, $a_1(K[S]) \geq \deg(\wbf) =  \delta_1(\wbf_j) - 1$.
\item Let $\ubf =  \wbf_i + (n_1d, n_2d) \in \Lcal'$. Since
$n_1 \le \delta_2(\wbf_i) - \deg(\wbf_i) - 1$ and $ n_2 \le \delta_1(\omega_1(i) + n_1d) - \deg(\wbf_i) - n_1 -1$, we have
\begin{align*}
    \deg(\ubf) & = \deg(\wbf_i) + n_1 + n_2\\
&\leq \delta_1(\omega_1(i) + n_1d) - 1\\
&\leq  \delta_1(\omega_1(i))  + n_1 - 1\\
&\leq  \delta_1(\wbf_i)  + \delta_2 (\wbf_i) - \deg(\wbf_i) - 2,
\end{align*}
 which implies \((ii)\) by Theorem \ref{A4a}.
 \item Using a representation $\omega_1(i) = \sum_{j=1}^{k-1} \beta_j a_j$ with $\sum_{j=1}^{k-1} \beta_j = \delta_1(\omega_1(i))$, we get
 $\omega_1(i) \geq  \delta_1(\wbf_i) a_1$. Similarly, $\omega_2(d-i) \geq \delta_2(\wbf_i)  (d-a_{k-1})$. Hence
 $$\deg(\wbf_i) \geq  \delta_1(\wbf_i) \frac{a_1}{d} + \delta_2(\wbf_i) \frac{ d-a_{k-1}}{d}.$$
 Replacing $\deg(\wbf_i)$ in the inequality in \((ii)\) by the right hand quantity above, we then get \((iii)\).
\end{enumerate}
\end{proof}

Note that the upper bounds in the above theorem is at most two times of the lower bound. From this point of view, they are quite good. In Section \ref{Appl} we give an example to show that all bounds of the above theorem are sharp.  

Unfortunately, the computation of Ap\'ery sets and degrees of their elements are in general very complicated, so that in general we cannot derive from the above theorem explicit bounds in terms of $a_1,\ldots,a_k$.  In particular, we cannot derive even the Gruson--Lazarsfeld--Peskine bound in \cite{GLP} for monomial curves, which states that
$$\reg(K[S]) \leq \deg(C) - \codim(C) = d- k+1.$$
Recall that for a sequence of positive integers $\Acal= \{0 =: \alpha_0, \alpha_1, \alpha_2, \ldots , \alpha_k\}$ with $0 < \alpha_1 < \alpha_2 < \cdots < \alpha_k$ its $i$-th gap is the number $\lambda_i = \alpha_i - \alpha_{i-1} - 1$, where $1\leq i\leq k$ . Note that $d-k = \lambda_1 + \cdots + \lambda_k$.

If $\lambda_i \neq 0$, by abuse of terminology, we also call the set $[a_i+1, a_{i+1}-1]$ the $i$-th gap.

The largest gap $\lambda_{\max} := \max\{ \lambda_1,\lambda_2,\ldots,\lambda_k\}$. Assume that $\lambda_{\max} = \lambda_i$ for some $1\leq i \leq k$. Then, the second largest gap is  $\lambda_{\mathrm{sl}} := \max\{ \lambda_1,\ldots,\lambda_{i-1},\lambda_{i+1},\ldots,\lambda_k\}$.
L'vovsky \cite{Lv} provided a much better bound than the one of \cite{GLP}: 
$$\reg(K[S]) \leq \lambda_{\max} + \lambda_{\mathrm{sl}} + 1.$$
In the sequel,  under some special  assumptions, we give better bounds than the above L'vovsky's bound.

First, we can reprove \cite[Theorem 2.7 and Proposition 3.4]{HeHS}.
\begin{Corollary} \label{A5} 
\begin{enumerate}[\((i)\)]
    \item Assume that $a_1=1$ and $a_{k-1} = d-1$. Let $\varepsilon = \max\{i|\ 1,\ldots,i, d-1,\ldots,d-i \in \{a_1,\ldots,a_{k-1} \}\}$ and $\lambda_{\max}$ be the largest gap of $A_1$. Then
$$\reg(K[S]) \leq \left\lfloor \frac{\lambda_{\max} -1}{\varepsilon} \right\rfloor + 2.$$
\item Assume that $a_1=1,\ldots,a_p = p, a_{p+1} \ge p+2$ for some $p\ge 1$. Denote the first non-zero gap of $A_1$ by $\lambda = a_{p+1}-p-1$. Then
$$\reg(K[S]) \geq \left\lceil \frac{\lambda}{p} \right\rceil + 1 = \left\lfloor \frac{\lambda -1}{p} \right\rfloor + 2.$$
In particular,  the bound in \((i)\) is attained if $\lambda = \lambda_{\max}$ and $ \varepsilon = p$.
\end{enumerate}
\end{Corollary}

\begin{proof} 
\begin{enumerate}[\((i)\)]
    \item Since $1\in A_1, A_2$, $\omega_1(i) = \omega_2(i) = i$, whence $\wbf_i = (i, d-i)$ for all $i\le d$. In particular $\deg(\wbf_i) = 1$. By Lemma \ref{B1}, $a_2(K[S]) + 2 =1$. So, in order to prove \((i)\), it is left to estimate $a_1(K[S])$ under the assumption that $K[S]$ is not a Cohen--Macaulay ring. Let $x\in \Iset$. There is $i$ such that $a_i < x< a_{i+1}$. Since $x = a_i + x -a_i$, and $1,\ldots,\varepsilon \in A_1$, by Remark \ref{degree}(2), $\delta_1(x) \leq 1 + \lceil (x-a_i)/\varepsilon\rceil$. Similarly, $\delta_2(d-x)\leq 1 + \lceil (a_{i+1} - x)/\varepsilon\rceil$. Hence, by Theorem \ref{A4}\((ii)\)
\begin{align*}
    a_1(K[S]) + 1 &\leq  \left\lceil \frac{ x-a_i}{\varepsilon} \right\rceil + \left\lceil \frac {a_{i+1} - x} {\varepsilon} \right\rceil   \\
&\leq  \frac{ x-a_i}{\varepsilon} + 1 - \frac{1}{\varepsilon} +  \frac{a_{i+1} - x}{\varepsilon} + 1 - \frac{1}{\varepsilon} = \frac{\lambda_{\max} - 1}{\varepsilon} + 2.
\end{align*}
Hence, $a_1(K[S]) + 1\leq \lfloor \frac{\lambda_{\max} - 1}{\varepsilon} \rfloor + 2$, as required.
\item Under the setting, it is clear that $\omega_1(i) = i$ for all $i\leq d-1$ and $\delta_1(a_{p+1} -1) \geq 1+ \lceil \lambda/ p \rceil$. By Theorem \ref{A4}\((i)\), $a_1(K[S]) + 1 \geq 1+ \lceil \lambda/ p \rceil$. It is clear that  $\lceil \lambda/ p \rceil = \lfloor (\lambda -1)/ p \rfloor + 1$.
\end{enumerate}
\end{proof}

If $K[S]$ is a Cohen--Macaulay ring, using the Lemma \ref{FrobSum} and Lemma \ref{B2} (or better Lemma \ref{B1}), we get a good bound on $\reg(K[S])$.  For an example, all Cohen--Macaulay rings in Proposition \ref{A3Add} have regularity $2$.

 If $K[S]$ is a non-Cohen--Macaulay Buchsbaum ring, then by Theorem \ref{A3New}, we get $a_1(K[S])  + 1  = \max_{i\in \Iset}\{\deg(\wbf_i)\} + 1$. Hence by Lemma \ref{B1}, $\reg(K[S]) \leq a_2(K[S])  + 3.$ This is a particular case of a more general result of \cite[Corollary 2.8]{HM}, which states that for any Buchsbaum graded $K$-algebra $R$ of dimension $d$, $\reg(R) \leq  a_d(R) + d + 1$.

From this observation, using Lemma \ref{B2} and Lemma \ref{FrobSum}, we immediately get:

\begin{Corollary} \label{BbmReg1} Let $\epsilon = \max\{a_1, d-a_{k-1}\} $.  Assume that $K[S]$ is a  Buchsbaum ring, then 
\begin{enumerate}[\((i)\)]
    \item $\reg(K[S]) \leq a_1 + (d-a_{k-1}).$
    \item If,  moreover, there are $i, i+1 \in A_1$ for some $0\leq i \leq d-1$, then  $\reg(K[S]) \leq \epsilon +1$.
    \item If, moreover, there are $i, i+1,\ldots, i+\epsilon -1  \in A_1$ for some $0\leq i \leq d-\epsilon + 1$, then  $\reg(K[S]) \leq 3$.
\end{enumerate}
\end{Corollary}

The main goal of \cite{T} is to determine when $K[S]$ is a Cohen--Macaulay or Buchsbaum ring in terms of $a_1,\ldots,a_k=d$. We already discussed the Cohen--Macaulay property in the previous Section \ref{Locoh}. Concerning the Buchsbaum property the problem was completely solved  in the case $a_1 = d-a_{k-1} = 1$, see \cite[Theorem 4.1 and Theorem 4.3]{T}. In a recent paper \cite{LT}, a new characterization in terms the Castelnuovo--Mumford regularity was given. In the following consequence, using the above study, we can give a very simple sufficient criterion for $K[S]$  to be a Buchsbaum ring, and quickly derive some results from \cite{LT}. Moreover, we see that the implication $ (3) \Rightarrow (1)$ in \cite[Theorem 2.4 and Theorem 2.7]{LT} holds without the assumptions there. 

\begin{Corollary} \label{BbmReg2}
\begin{enumerate}[\((i)\)]
    \item If $\reg(K[S]) = 2$, then $K[S]$ is a Buchsbaum ring.
    \item (\cite[Theorem 2.3]{LT}) If $ a_1 =1$ and $ a_{k-1} = d-1$, then $K[S]$ is a  Buchsbaum ring if and only if $\reg(K[S]) = 2$.
    \item Assume that $a_1=1$ and $a_{k-1} = d-1, \ k<d$. Let $\varepsilon = \max\{i\mid 1,\ldots,i, d-1,\ldots,d-i \in \{a_1,\ldots,a_{k-1} \} \}$. If  $\lambda_{\max} \leq \varepsilon$, then $K[S]$ is a Buchsbaum ring.
\end{enumerate}
\end{Corollary}

\begin{proof} 
\begin{enumerate}[\((i)\)]
    \item Assume that $\reg(K[S]) = 2$. W.l.o.g we may assume that $K[S]$ is not a Cohen--Macaulay ring. Then $a_1(K[S]) = 1$. From Lemma \ref{A1}\((ii)\) and  Lemma \ref{A3}, $S'\setminus S = \{\wbf_i\mid i\in \Iset\}$ whence $\deg(\wbf_i) = 1$ for all $i\in \Iset$. By Theorem \ref{A3New}\((iii)\),  $K[S]$ is a Buchsbaum ring.
    \item If $a_1 = d-a_{k-1} = 1$ and $K[S]$ is a Buchsbaum ring, then $\reg(K[S]) = 2$ follows from Corollary 
\ref{BbmReg1}\((ii)\). The converse follows from \((i)\).
\item Since $k<d$ and $a_1 = d-a_{k-1} = 1$, $K[S]$   is not Cohen-Macaulay. Under the assumption, by Corollary \ref{A5}, $\reg(K[S]) \leq 2$, whence $\reg(K[S]) = 2$.  Hence, the statement follows from \((ii)\).
\end{enumerate}
\end{proof}

Rewrite $ A_1 = \{a_0,\ldots,a_k\} = \bigsqcup_{j=0}^r [b_{2j},b_{2j+1}]$, where $0=:b_0 \leq b_1 < b_2 \leq b_3 < \cdots < b_{2r}\leq b_{2r+1} := d$ and $b_{2j-1} + 2 \leq b_{2j}$ for $j=1,\ldots,r$. Note that $r$ is the number of non-zero gaps of $A_1$. Since we already assume that $k<d$, $r\geq 1$. The case $r=1$ can be solved below. 

\begin{Corollary} \label{R1}  Assume that $A_1 = [0,p] \cup [q,d]$ with $p\geq 1$ and $p+2\leq q\leq d$.
\begin{enumerate}[\((i)\)]
    \item If $q<d$ we may further  assume that $p \leq  d-q$. Then $\reg(K[S]) =\left\lceil \frac{q-1}{p} \right\rceil$.
    \item If $q = d$, then $\reg(K[S]) \in \{a+1, a+2\}$, where 
$$a = \left\lfloor \frac{d}{p} \right\rfloor - \left\lceil \frac{p\lfloor \frac{d}{p} \rfloor + 2}{d} \right\rceil.$$
\end{enumerate}
\end{Corollary}

\begin{proof} 
\begin{enumerate}[\((i)\)]
    \item follows from Corollary \ref{A5}\((ii)\).
    \item By \cite[Corollary 3.4]{T}, $K[S]$ is Cohen--Macaulay. Hence we only need to estimate $a_2(K[S]) + 2$. Since $1\in A_1$, $F_1 = -1$ and $\omega_1(j) = j$. By Brauer's bound \cite[Theorem 7]{Bra} (see also Theorem 3.3.1 in \cite{Ra} and discussion after that), we have
$F_2 = \left\lfloor \frac{d-2}{p} \right\rfloor (d-p) -1$. Hence, by Lemma \ref{B2}, 
\begin{align*}
    a_2(K[S]) \leq  \left\lfloor \frac{F_2-1}{d} \right\rfloor &\leq \left\lfloor \frac{ \lfloor \frac{d}{p}\rfloor (d-p) -2}{d} \right\rfloor \\
&= \left\lfloor \left\lfloor \frac{d}{p} \right\rfloor -  \frac{p \lfloor \frac{d}{p}\rfloor + 2}{d} \right\rfloor =  \left\lfloor \frac{d}{p} \right\rfloor - \left\lceil \frac{p \lfloor \frac{d}{p}\rfloor + 2}{d} \right\rceil = a.
\end{align*}
On the other hand, assume that $F_2 \equiv (d-j)\bmod d$ for some $1 \geq j \leq d-1$. Then $\omega_2(d-j) = F_2 + d$. By Lemma \ref{B1},
$$a_2(K[S]) \geq \deg(\wbf_j) = \left\lceil \frac{F_2 - 1 + d+j}{d} \right\rceil - 2 \geq a-1,$$
as required.
\end{enumerate}
\end{proof}

Let $r\geq 2$. In \cite[Section 3]{LT}, under some very special assumptions, there are some very good bounds established for non-smooth curves. For an example, \cite[Theorem 3.4]{LT} states that if $b_1 = 0, b_{2r} < d$ and $2b_2 - 1 \leq b_3,\ r\geq 2$, then
$$\reg(K[S]) \leq \left\lfloor \frac{\lambda_{\max} - 1}{\varepsilon'} \right\rfloor + 3,$$
where $\varepsilon' = \min\{b_3, d- b_{2r} \}$. Using Theorem  \ref{A4}, we can reprove this result. However, the computation is still not  simple, so that we do not present it here. Instead, we consider the case of possibly small $b_3 - b_2$.

\begin{Proposition} \label{Bd2} Let $A_1 = \{a_0,\ldots,a_k\} = \bigsqcup_{j=0}^r [b_{2j},b_{2j+1}]$, where $0=:b_0 \leq b_1 < b_2 \leq b_3 < \cdots < b_{2r}\leq b_{2r+1} := d$ and $b_{2j-1} + 2 \leq b_{2j}$ for $j=1,\ldots,r$, $r\ge 2$.  If $b_1 = 0, b_{2r} < d$ and $b_5 - b_4 + 1\geq b_2$, then
$$\reg(K[S]) \leq 2+ \max\left\{ \left\lfloor \frac{ \lambda_{\mathrm{sl}}}{d-b_{2r}} \right\rfloor + \left\lfloor \frac{\lambda_{\max}}{b_2} \right\rfloor;\ \left\lfloor \frac{ \lambda_{\max}}{d-b_{2r}} \right\rfloor + \left\lfloor \frac{\lambda_{\mathrm{sl}}}{b_2} \right\rfloor \right\}.$$
\end{Proposition} 

\begin{proof} Let
$$M:=  \max\left\{ \left\lfloor \frac{ \lambda_{\mathrm{sl}}}{d-b_{2r}} \right\rfloor + \left\lfloor \frac{\lambda_{\max}}{b_2} \right\rfloor;\ \left\lfloor \frac{ \lambda_{\max}}{d-b_{2r}} \right\rfloor + \left\lfloor \frac{\lambda_{\mathrm{sl}}}{b_2} \right\rfloor \right\}\geq 0.$$
From the assumption we see that $1\in A_2$. Hence $F_2 = -1$ and $\omega_2(d- x) = d- x$ for all $x\leq d$. Since $b_5 - b_4 + 1\geq b_2$, $b_1= 0$,  we have $1 \leq F_1 \leq b_4 -1$ and
$$\omega_1(x) = \begin{cases} x & \text{if}\quad  b_4 \leq x < d \ \text{or} \ x \in \langle [b_2,b_3] \rangle \ \text{and}\ x< b_4,\\
x+ d &\text{otherwise}.
\end{cases}$$
 By Lemma \ref{B2}, $a_2(K[S]) + 2 \leq \lfloor \frac{F_1+ F_2}{d}\rfloor + 2 = 2$. So, it is left to  show that $a_1(K[S]) \leq M + 1 $, provided  that $\Iset\neq \emptyset$. By Theorem \ref{A4}\((ii)\) there is  $x\in \Iset$ such that
 $$a_1(K[S]) \leq \delta_1(\wbf_x ) + \delta_2(\wbf_x) - \deg(\wbf_x) - 2 =: D(x).$$
 It is enough to show that $D(x) \leq M + 1$. Let $\epsilon' := \min\{ b_2 , d-b_{2r}\}$. We distinguish three cases.
 \begin{enumerate}[\(\textbf{\text{Case}}\ \)\bf a.]
     \item $ x>b_5$. In this case $\wbf_x = (x, d-x)$, whence $\deg(\wbf_x) = 1$.  There is $i$ such that $b_5\leq a_i < x < a_{i+1}$. In the proof of Corollary \ref{A5}\((i)\) we know that 
 \begin{equation}\label{EBd21}
\delta_2(\wbf_x) \leq 1 + \left\lceil \frac{a_{i+1} - x}{d- b_{2r}} \right\rceil \leq 1+ \left\lceil \frac{a_{i+1} - x}{\epsilon'} \right\rceil.
\end{equation}
  Note that $x = b_4 + (x-b_4)$. If $x-b_4 \in A_1$, then $\delta_1(\wbf_x) = \delta_1(x) = 2$. We now assume that $x-b_4 \not\in A_1$
  \begin{enumerate}[\((\text{a}1)\)]
      \item If $ x-b_4 > a_i$, then using $x= b_4 + x-b_4 = (b_4 + b'_2) + a_i +  \left\lfloor \frac{x-b_4 - a_i}{b_2} \right\rfloor b_2$, where $0\leq b'_2 < b_2$, and noticing that $b_4 + b'_2 \leq b_5$, whence $b_4+b'_2 \in A_1$, we can conclude that 
  \begin{equation}\label{EBd22}
  \delta_1(\wbf_x) = \delta_1(x) \leq 2 + \left\lfloor \frac{x-b_4 - a_i}{b_2} \right\rfloor  \leq 2+ \left\lfloor \frac{x-b_4 - a_i}{\epsilon'} \right\rfloor.
  \end{equation}
 From (\ref{EBd21}) and (\ref{EBd22}) we get
  \begin{equation}\label{EBd23} 
  D(x)  \leq \left\lceil \frac{a_{i+1} - x}{\epsilon'} \right\rceil +  \left\lfloor \frac{x-b_4 - a_i}{\epsilon'} \right\rfloor \leq \left\lceil \frac{\lambda_i -3 }{\epsilon'} \right\rceil  \leq 1 + \left\lfloor \frac{\lambda_i  -4}{\epsilon'} \right\rfloor \leq M+1.
   \end{equation}
      \item $a_j < x-b_4 < a_{j+1}$ for some $j<i$, that is $x$ and $x-b_4$ belong to two different gaps of $A_1$. Using $x= b_4 + x-b_4 = (b_4 + b^{''}_2) + a_j +  \left\lfloor \frac{x-b_4 - a_j}{b_2} \right\rfloor b_2$, where $0\leq b^{''}_2 < b_2$, we get
  \begin{equation}\label{EBd24}
  \delta_1(\wbf_x) = \delta_1(x) \leq 2 + \left\lfloor \frac{x-b_4 - a_j}{b_2} \right\rfloor \leq 2 + \left\lfloor \frac{a_{j+1} - 1 - a_j}{b_2} \right\rfloor \leq 2+ \left\lfloor \frac{\lambda_j}{b_2} \right\rfloor.
   \end{equation}
   From (\ref{EBd21}) we also have $\delta_2(\wbf_x) \leq 1 + \left\lceil \frac{ \lambda_i}{d-b_{2r}} \right\rceil$. Hence,
    \begin{equation}\label{EBd25}
  D(x)  \leq \left\lceil \frac{ \lambda_i}{d-b_{2r}} \right\rceil + \left\lfloor \frac{\lambda_j}{b_2} \right\rfloor \leq 1+  \left\lfloor \frac{ \lambda_i}{d-b_{2r}} \right\rfloor + \left\lfloor \frac{\lambda_j}{b_2} \right\rfloor.
     \end{equation}
   Note that if $\lambda_k\ge \lambda_l$ are two different gaps, then $\lambda_k \leq \lambda_{\max}$ and $\lambda_l \leq \lambda_{\mathrm{sl}}$. This implies
   $\left\lfloor \frac{ \lambda_i}{d-b_{2r}} \right\rfloor + \left\lfloor \frac{\lambda_j}{b_2} \right\rfloor \leq  M.$ Using this remark, from (\ref{EBd23}) and (\ref{EBd25}) we get in Case $a$ that $ D(x) \leq  M +1$.
  \end{enumerate}
     \item $ b_3 < x< b_4$.
     \begin{enumerate}[\((\text{b}1)\)]
         \item $x\not\in \langle [b_2,b_3]\rangle$. Then $\omega_1(x) = x+d$, whence $\wbf_x = (d+x, d-x)$ and $\deg(\wbf_x) = 2$. Assume that $b_3 = a_i$. Then $a_{i+1} = b_4$, and $x$ belongs to the gap $\lambda_i$. Write $d+x = b_4 + (d+x-b_4)$. Note that $b_3 <  d+x - b_4 < d$. If $d+x-b_4 \in A_1$, then $ \delta_1(\wbf_x) = \delta_1(d+ x)  =2$. Assume now that $d+x-b_4\not\in A_1$.
  
   If $d+x-b_4$ belongs to $j$-th gap with $i\neq j$, then similar to the subcase \(\text{(a2)}\), we can conclude  that 
  $$\delta_1(\wbf_x) = \delta_1(d+ x)  \leq 2+ \left\lfloor \frac{\lambda_j}{b_2} \right\rfloor.$$
  On the other hand,  using the relation $d-x = (d-b_4) + (b_4-x)$,  as shown in the proof of Corollary \ref{A5}\((i)\), 
   \begin{equation}\label{EBd26}
   \delta_2(\wbf_x)  \leq 1 +  \left\lceil \frac{b_4-x}{d-b_{2r}} \right\rceil \leq  1+ \left\lceil \frac{\lambda_i}{d-b_{2r}} \right\rceil \leq 2+ \left\lfloor \frac{\lambda_i -1}{d-b_{2r}} \right\rfloor.
    \end{equation}
   Hence,  the argument at the end of Case \text{(a2)} gives  $D(x) \le M$.
  
  The remaining case is  $b_3< d+x-b_4 < b_4$; that is both $x$ and $d+x-b_4$ belong to the same $i$-th gap. Using the relation $d+x = b_3 + b_4 + (d+x- b_3-b_4)$, similar to the subcase (a1) we get
  $$\delta_1(\wbf_x) = \delta_1(d+ x)  \leq 2+ \left\lfloor \frac{d+x-b_3-b_4}{b_2} \right\rfloor \leq 2+ \left\lfloor \frac{d+x-b_3-b_4}{\epsilon'} \right\rfloor.$$
  On the other hand $d-x = 2(d-b_4) + (b_4 - (d+x-b_4))$. Note that $1, d-b_4 \in A_2$. As shown in the proof of Corollary \ref{A5}\((i)\), 
   $$\delta_2(\wbf_x)  = \delta_2(d-x) \leq 2 + \left\lceil \frac{b_4 - (d+x-b_4)}{d-b_{2r}} \right\rceil \leq 2 +  \left\lceil \frac{2b_4 - (d+x)}{\epsilon'} \right\rceil.$$
   Hence,
   \begin{align*}
       D(x) &\leq \left\lfloor \frac{d+x-b_3-b_4}{\epsilon'} \right\rfloor + 
 \left\lceil \frac{2b_4 - (d+x)}{\epsilon'} \right\rceil \\
&\leq \left\lceil \frac{b_4 - b_3}{\epsilon'} \right\rceil = \left\lfloor \frac{\lambda_i}{\epsilon'} \right\rfloor + 1 \leq M+ 1.
   \end{align*}
         \item $x\in \langle [b_2,b_3]\rangle$. Then $\omega_1(x) =x$ and $\deg(\wbf_x) = 1$. By (\ref{EBd26}),
   $$\delta_2(\wbf_x)  \leq 1+ \left\lceil \frac{b_4 - x}{d-b_{2r}} \right\rceil \leq 1 +  \left\lceil \frac{b_4 - x}{\epsilon'} \right\rceil.$$
   Since $x< b_4$, we can find non-negative integers $n_j$ such that $x =  \sum_{j=b_2}^{b_3} n_jj$ such that $\delta_1(x) =  \sum_{j=b_2}^{b_3} n_j$. Assume that $\delta_1(x) \geq 3$. Let $b$ be a subsum of two summands in the sum $ \sum_{j=b_2}^{b_3} n_jj$ (could be the same). Then $b\geq b_3+1$, because otherwise replacing this subsum by $b$ would give a presentation of $x$ with the sum of coefficients strictly smaller than $\delta_1(x)$, which is impossible. We have $x-b\geq (\delta_1(x) - 2)b_2$. Hence,
   $$\delta_1(\wbf_x) = \delta_1(x) \leq  2+ \left\lfloor \frac{x-b}{b_2} \right\rfloor \leq  2+ \left\lfloor \frac{x-b_3-1}{b_2} \right\rfloor \leq 2+ \left\lfloor \frac{x-b_3-1}{\epsilon'} \right\rfloor.$$
   This implies
   $$D(x) \leq   \left\lceil \frac{b_4 - x}{\epsilon'} \right\rceil + \left\lfloor \frac{x-b_3 -1}{\epsilon'} \right\rfloor \leq  \left\lceil \frac{b_4-b_3 -1}{\epsilon'} \right\rceil = \left\lfloor \frac{\lambda_i}{\epsilon'} \right\rfloor \le M.$$
     \end{enumerate}
     \item $0< x< b_2$, i.e. $x$ belongs to the first gap (note that $a_0 = 0, \ a_1 = b_2$). Then $x\not\in \langle A_1\rangle$, whence $\omega_1(x) =d+x$, $\wbf_x = (d+x, d-x)$ and $\deg(\wbf_x) = 2$. We have $d+ x = b_2 + (d+x-b_2)$. If $d+x-b_2$ belongs to the first gap too, then $d+x-b_2 < b_2$. This implies $d< 2b_2 - 1< b_5$, a contradiction. Hence if $d+x-b_2\not\in A_1$, then $x$ and $d+x-b_2$ belong to two different gaps of $A_1$. The argument in the case \text{(a2)} gives us $D(x)\leq M + 1$, as required. 
 \end{enumerate}
\end{proof}

In many consequences of Theorem \ref{A4} above, $A_1$ contains a sequence of consecutive integers. We now give another kind of situations.

\begin{Lemma} \label{A7} Let $\Acal =\{ \alpha_1< \alpha_2 < \cdots < \alpha_k\}\ (k\geq 2)$. Assume that $\alpha_i
 \not\equiv \alpha_j \bmod \alpha_k$ for all $i\neq j$. Then for all $1\leq i\leq d-1$ we have
$$\delta(\omega(i)) \leq 1+ \left\lfloor \frac{2\alpha_k \lfloor \alpha_1/k \rfloor + \lambda_{\max}}{\alpha_1} \right\rfloor  \leq 1 + \frac{2\alpha_k}{k} + \frac{\lambda_{\max}}{\alpha_1}.$$
\end{Lemma}

\begin{proof}  The second inequality immediately follows from the first one. Let $F:= F(\alpha_1,\ldots,\alpha_k)$. There is an unique $1\le j \le \alpha_k$ such that $\omega(i) \equiv F+j \pmod{\alpha_k}$. Since $F$ is a largest integer not belonging to $\langle \Acal \rangle$, $F+j \in \langle \Acal \rangle$ and
$\omega(i) = F+j - p\alpha_k$ for some $p\in \Nset$. If $p\ge 1$, then $\omega(i) \leq F$. Note that for any $n\in \langle \Acal \rangle$, $\delta(n) \leq \frac{n}{\alpha_1}$. By Lemma \ref{Selmer1}, $F\le 2\alpha_k \lfloor \alpha_1/ k \rfloor$. Hence 
$$\delta(\omega(i)) \leq \frac{F}{\alpha_1} \leq \frac{2\alpha_k \lfloor \frac{\alpha_1}{k}\rfloor}{\alpha_1} -1.$$
Next, assume that $\omega(i) = F+j$. Let $l$ be the largest index such that $\alpha_l \le j$. Then $j-\alpha_l \leq \alpha_{l+1} - 1 - \alpha_l = \lambda_l$ and
$\omega(i) = F+ j = F+ \varepsilon + 1\cdot \alpha_l$, where $0\leq \varepsilon \leq \lambda_l$. Hence
$$ \delta(\omega(i)) \leq 1 + \delta(F+ \varepsilon) \leq  1+ \frac{F+ \varepsilon}{\alpha_1} \leq  1+ \frac{F+ \lambda_l}{\alpha_1} \leq 1+ \frac{F+ \lambda_{\max}}{\alpha_1}.$$
Using again the above bound on $F$, we get
$$ \delta(\omega(i)) \leq 1 +  \frac{2\alpha_k \lfloor \alpha_1/k \rfloor + \lambda_{\max}}{\alpha_1} ,$$
which implies the first inequality.
\end{proof}

We can now give an explicit upper bound on $a_1(K[S])$.

\begin{Theorem} \label{A9} Assume that $a_i
 \not\equiv a_j \pmod{a_1}$ and  $a_i
 \not\equiv a_j \pmod{d -a_{k-1}}$ for all $i\neq j$. Then
 $$a_1(K[S]) <  \frac{2d+\lambda_{\max}}{k}\left( 1+\frac{a_{k-1}-a_1}{d} \right),$$
 where $\lambda_{\max}$ is the largest gap of the sequence $a_1,\ldots,a_{k-1}, a_k=:d$.
 
 As a consequence, we have
 $$\reg(K[S]) \leq 1+  \frac{2d+\lambda_{\max}}{k}\left( 1+\frac{a_{k-1}-a_1}{d} \right).$$
\end{Theorem}

\begin{proof} By the assumption, it follows that $a_1 \geq k$ and $d- a_{k-1} \geq k$. Note that the largest gap of the sequence $d- a_{k-1},\ldots, d - a_1, d$ is also the same $\lambda_{\max}$. By Lemma \ref{A7}, we then get $\delta_1(\wbf_i) \leq 1+  \frac{2d + \lambda_{\max}}{k}$ and $\delta_2(\wbf_i) \leq 1+ \frac{2d + \lambda_{\max}}{k}$. Hence, by Theorem \ref{A4}\((iii)\), we get
\begin{align*}
    a_1(K[S]) & \leq \left\lfloor \left( 1- \frac{a_1}{d} \right)  \max_{i\in \Iset} \{\delta_1(\omega_1(i)) \} +  \frac{a_{k-1}}{d}  \max_{i\in \Iset}
\{  \delta_2(\omega_2(d-i)) \} \right\rfloor  - 2\\
& \leq \left\lfloor 1 +   \frac{2d + \lambda_{\max}}{k} \right\rfloor \left( 1 - \frac{a_1}{d} + \frac{a_{k-1}}{d} \right) - 2 <  \frac{2d+\lambda_{\max}}{k}\left( 1+\frac{a_{k-1}-a_1}{d} \right),
\end{align*}
as required.

For the bound on $\reg(K[S])$, by Lemma \ref{Selmer1}, we have $F_1 < 2da_1/k$ and $F_2 < 2d(d-a_{k-1})/k$. Recall that $k\ge 3$. Since $a_{k-1} - a_1 \geq k-1$, by Lemma \ref{B2}, it implies
$$a_2(K[S]) +2 \leq \frac{F_1+F_2}{d}  +2 < 2\cdot\frac{d + a_1 - a_{k-1}}{k} < \frac{2d - 2k+4}{k} + 2  = \frac{2d+4}{k} .$$
On the other hand, note that  under our assumption $a_1\geq k$, whence $\lambda_{max} \geq a_1-a_0-1 \geq k-1$. This implies 
\begin{align*}
    \frac{2d+\lambda_{\max}}{k} \left( 1+\frac{a_{k-1}-a_1}{d} \right) &\geq \frac{2d}{k} + \frac{\lambda_{\max }+ 2(a_{k-1}- a_1)}{k} \\
&> \frac{2d + 3k-5}{k}  \geq a_2(K[S]) +2.
\end{align*}
Since $\reg([K(S)] = \max\{a_1(K[S]) + 1,\ a_2(K[S]) + 2\}$,  the bound on $\reg(K[S])$ follows from the above inequality and the inequality on $a_1(K[S])$.
\end{proof}
The above bound is less than $6d/k$. If $A_1$ has $r\geq 1$ non-zero gaps, then $\lambda_{\max} + \lambda_{\mathrm{sl}} \geq 2(d-k)/r$. It is immediate to check that if $k\geq 4r$ and $d\geq 4k$, then $6d/k \leq 2(d-k)/r$. Thus, if $k\geq 4r$ and $d\geq 4k$, the bound in Theorem \ref{A9} is better than L'vovsky's bound.

\section{Application and Examples} \label{Appl}

Note that  from the proof of  Corollary \ref{A5}\((i)\), the upper bound there for smooth curves is exactly the upper bound in Theorem \ref{A4}\((ii)\). Hence by Corollary \ref{A5}\((ii)\), this upper bound is attained by smooth curves of the type $C(1,\ldots,p, a_{p+1},\ldots, a_l, d-q, d-q+1,\ldots,d)$, where  $q\geq p$, $a_{p+1} \geq p+2$ and $a_l < d-q$. Below we give an example of non-smooth curves for which all  bounds of Theorem \ref{A4} are attained. In this example, the first bound of Lemma \ref{A7} is also attained.

\begin{Example} \label{C1} Given $k\ge 2$, let $a_1 = 2k-1,\ a_2 = 2a_1+1 = 4k-1,\ a_3 = 2a_1 + 2,\ldots, a_k = 2a_1 +(k-1) = 5k-3$. This example satisfies the assumptions of Lemma \ref{A7} applied to $\Acal = \{a_1,\ldots,a_k\} = A_1\setminus \{0\}$ and attains its first upper bound, which is $6$.

All bounds in Theorem \ref{A4} are also attained by this example and they are equal to $5$.

{\rm  Indeed, we have $\lambda_{\max} = a_2 - a_1 - 1 = 2k-1 $. Hence, by  the first upper bound of Lemma \ref{A7}, for all $x< a_k$, 
\begin{equation}\label{EC11}
\delta_1(\wbf_x) = \delta_{\Acal}(\omega_{\Acal}(x)) \leq  1+  \left\lfloor \frac{2 (5k-3) \lfloor (2k-1)/k \rfloor + 2k-1}{2k-1} \right\rfloor = 1+ \left\lfloor \frac{12k-7}{2k-1} \right\rfloor = 6.
\end{equation}
We have $a_k + 2k = 7k-3 = 3a_1 + k < 4a_1$. Since $a_2,\ldots,a_k > 2a_1$, if $a_k + 2k = \mu_1a_1 + \cdots + \mu_k a_k$, then $\mu_2 + \cdots + \mu_k \leq 1$. If $\mu_2 + \cdots + \mu_k = 1$, then $2a_1 + 1 = a_2 \leq \mu_2 a_2 + \cdots + \mu_k a_k \leq a_k = 2a_1 + (k-1)$, whence $a_1+1 \leq \mu_1 a_1 \leq a_1 + (k-1) < 2 a_1$, a contradiction. The case $\mu_2 + \cdots + \mu_k = 0$ gives $\mu_1 a_1 = 3a_1 + k$, impossible. So $2k + a_k \not\in \langle \Acal \rangle$. On the other hand, $2k+ 2a_k = 10k-6 = 6a_1$. Hence $\omega_{\Acal} (2k) = 2k+2a_k$ and $\delta_{\Acal}(2k) \leq 6$. Assume that $6a_1 = \mu_1a_1 + \cdots + \mu_k a_k$ with $\mu:= \mu_2 + \cdots + \mu_k > 0$.
Then $2\mu a_1 + \mu = \mu a_2  \leq 6a_1 $ implies $\mu \leq 2$. Since $2\mu a_1 + \mu = \mu a_2  \leq \mu_2 a_2 + \cdots + \mu_{k-1}a_{k-1} \leq \mu a_{k-1} = 2\mu a_1 + \mu (k-2) $, we have 
$$(6-2\mu)a_1 - \mu(k-2) \leq \mu_1 a_1 \leq (6-2\mu) a_1 - \mu.$$
This implies 
$$6- 2\mu - \frac{\mu(k-2)}{2k-1} \leq \mu_1 \leq \left\lfloor 6-2\mu - \frac{\mu}{2k-1} \right\rfloor = 5-2\mu.$$
A direction computation with $\mu = 1,2$ gives a contradiction. Therefore $\mu= 0$ and $\delta_{\Acal}(2k) = 6$.

We now compute $a_1(K[S])$. Since $1\in A_2$, $\omega_2 (x) = x$ for all $x< a_k$. Using the argument in the proof of Corollary \ref{A5}\((i)\), we get $\max_{1\leq x\leq a_k} \delta_2(\wbf_x) = \delta_2(\wbf_{3k-3}) = 4$. Hence, by Theorem \ref{A4}\((iii)\)
$$a_1(K[S]) \leq  \left\lfloor 6\left( 1- \frac{2k-1}{5k-3} \right)  + 4\cdot\frac{5k-4}{5k-3} \right\rfloor - 2 = 5.$$
On the other hand, by the above computation $\wbf_{2k} = (12k-6,3k-3) \in \Iset$, since $\delta_1(\wbf_{2k}) = 6 > \deg(\wbf_{2k}) = 3$. Taking $n_1 = 0$ and $n_2 = \delta_1(\wbf_{2k}) - \deg(\wbf_{2k}) - 1 = 2$ in Theorem \ref{A4a}, we get $a_1(K[S]) \geq  \deg(\wbf_{2k} + 2\ebf_{a_{2k}}) = 5$. Hence $a_1(K[S])  = 5$. This  is exactly the lower bound in Theorem \ref{A4}\((i)\) and the upper bound in Theorem \ref{A4}\((iii)\). Since the statement \((iii)\) follows from \((ii)\) in Theorem \ref{A4}, the bound in \((ii)\) is also attained.

By Lemma \ref{Selmer1}, $F_1 \leq 2a_k - a_1$. Since $1\in A_2$, $F_2 = -1$. Hence, by Lemma \ref{B2}, $a_2(K[S]) \leq 1$, whence $\reg(K[S]) = 6$.
}
\end{Example}

We are now applying Algorithm 2 to calculate some examples. The next example was considered in \cite[Example 4.3]{BGi}, where $a_1(K[S]), \ a_2(K[S]$ and $\reg(K[S])$ were computed.

\begin{Example} \label{ExAl1} {\rm
Let $C= C(5,9,11,20)$. Algorithm 2 gives the result in Table 1. Here, for each $i \neq 5,9,11,20$ we give the vector $\wbf_i$, its degree, the degrees $(\delta_1,\delta_2)$ of its component. In the column $\Iset$, the value 1 means that this $i\in \Iset$, otherwise it is left blank.

From Table 1 we can see that $a_2(K[S]) = \deg(\wbf_{17}) - 2 = 3$ and $\Iset = \{ 4\}$. In particular $K[S]$ is not Cohen--Macaulay. Since $\delta_1(\wbf_4) = \delta_2(\wbf_4) = 4 = \deg(\wbf_4) + 1$, the set $ \Lcal = \{\wbf_4\}$. Hence by Theorem \ref{A4a}, $a_1(K[S]) = 3$,  $\reg(K[S]) = 5$ and $\ell(H^1_{\mfr}(K[S]) = 1$.

This example is interesting in some senses. 

First, by Theorem \ref{A3New}\((iii)\), $K[S]$ is a Buchsbaum ring. In the proof of  \cite[Theorem 4.1]{BCFH}, it was shown that for any curve in $\Pset^3$, $a_2(K[S]) + 2 \leq a_1(K[S]) + 1$, or equivalently, $\reg(K[S])$ is always attained at the first local cohomology. However, the current example shows that this property already  does not hold  for Buchsbaum curves in $\Pset^4$.

\centering \textbf{Table 1. $C(5,9,11,20)$.}
\begin{center}
    \begin{tabular}{|l|c|c|c|c|| l|c|c|c|c|}
\hline
$i$ & $ \wbf_i$ & $\deg$ & $(\delta_1,\delta_2)$ & $\Iset $ & $i$ & $ \wbf_i$ & $\deg$ & $(\delta_1,\delta_2)$ & $\Iset $ \\
\hline
1& (21,39) & 2 &(3,3)&  & 12& (32,48) & 4 & (4,4)&  \\
2 & (22,18) & 2 & (2,2)&  &13& (33,27)& 3& (3,3) & \\
3 & (23,37) & 3& (3,3) &  &14& (14,26) & 2& (2,2)& \\
4 &(24,36)& 3 &(4,4)&  4 &15& (15,45)& 3&(3,3) & 4\\
6 & (26,54) &4& (4,4)& & 16 &(16,24)& 2&(2,2)& \\
7& (27,33)& 3& (3,3) &  & 17 & (37,63)& 5& (5,5) & \\
8 & (28,52)& 4& (4,4) &  & 18&(18,22)& 2& (2,2) & \\
10& (10,30)& 2& (2,2)&  & 19& (19,41)& 3 & (3,3)& \\
\hline
\end{tabular}
\end{center}

Secondly, as mentioned before Corollary \ref{BbmReg2}\((ii)\), for Buchsbaum curves we have $\reg(K[S])$ only takes one of two values $\{a_2(K[S]) + 2,\ a_2(K[S]) + 3 \}$. By Lemma \ref{B2}, all Buchsbaum curves satisfying Corollary \ref{BbmReg2}\((i)\) have $a_2(K[S]) = -1$ and $\reg([K(S)]) = 2 = a_2(K[S]) + 3$. The current example shows that $\reg([K(S)]) = a_2(K[S]) + 2$ can happen.
} 
\end{Example}

In the following example we give more curves in $\Pset^4$ with $ a_2(K[S]) > a_1(K[S])  $. These curves satisfy the condition (\ref{EA3N}) in Theorem \ref{A3New}, which  however can be  Buchsbaum or not.

\begin{Example} \label{ExAl2} {\rm 
The first example in Table 2 is the curve $C_3$ considered in Example \ref{A3c1}.

\centering \textbf{Table 2}
\begin{center} 

    \begin{tabular}{|c|c|c|c|c|c|c|}
\hline
$C(a_1,\ldots,a_k)$ & $\Iset$ & Buchsbaum& $\ell$ & $a_1$ & $a_2$ &$\reg$ \\
\hline
(1,4,21,85) & 20,41,63,62,83& No & 4 & 7 & 8 &10\\
(1,5,11,46)& 9 & Yes & 1 & 4 & 5 & 7\\
(12,17,20,29) & 19& Yes & 1& 3 & 4 & 6\\
(45,46,65,121)& 73&  Yes& 1 & 6 & 9 & 11\\
(3,4,17,55) & 13,30,47 & No & 3 & 5 & 6 & 8\\
(4,10,21,61) & 44,48 & No &  2 & 5 & 6 & 8\\
(25,35,44,123) & 62, 106& No & 2 & 7 & 10 & 12\\
\hline
\end{tabular}
\end{center}
 }
 \end{Example}

In the above example we have $ \max\{ \delta_2(\wbf_i)- \deg(\wbf_i) -1|\ i\in \Iset\} = 0$ for all curves. By Theorem \ref{A4a}, using  Algorithm 2 we can already compute the invariants in Table 2. In the general case, by Theorem \ref{A4a}, we also need to compute $\delta_1(\omega_1(i) + n_1d)$ for certain $i$ and $n_1$. Note that $i$ and $n_1$ are already calculated in Algorithm 2. If $\Iset$ only consists of a few elements, one can use Integer Programming to calculate   $\delta_1(\omega_1(i) + n_1d)$, as mentioned in Remark \ref{degree}. Otherwise, we need  Algorithm 3 in the Appendix.

 \begin{Example} \label{ExAl3} {\rm
In this example, all curves have $ \max\{ \delta_2(\wbf_i)- \deg(\wbf_i) -1|\ i\in \Iset\} > 0$. So, it is non-Buchsbaum. The computation is done by Algorithm 3. 

\centering \textbf{Table 3}
\begin{center} 
    \begin{tabular}{|c|c|c|c|c|c|}
\hline
$C(a_1,\ldots,a_k)$ & $\Iset$ &  $\ell$ & $a_1$ & $a_2$ &$\reg$\\
 \hline
 (2,10,22,57) & 18& 2 & 5& 6 & 8\\
 (2,7,12,14)& 1,3,4,6,8,10,11,13 & 40& 5& 0 & 6\\
 (39,58,68,129,158)& $\sharp(\Iset)= 21$& 80 & 8& 9 &11\\
\hline
\end{tabular}
\end{center}}
\end{Example}

We conclude this paper by giving an application to the structure theory of sumsets. Let $A = \{0=: a_0 < a_1< \cdots < a_k =:d\}$  be a set of $k+1 \geq 4$  non-negative
relatively prime integers.  The $h$-fold sumset  is the set $hA$.  Sumsets play a central role in the Additive Number Theory, see \cite{Na}. 

The following result is called the fundamental structure theorem in Additive Number Theory.

\begin{Theorem}
 {\rm \cite[Theorem 1.1]{Na}} \label{NaTh11}  There are integers $c_1, c_2$  and sets $C_i \subseteq [ 0, c_i-2]$ ($i=1,2$)  such that
$$hA = C_1 \sqcup [c_1,  h d - c_2] \sqcup (h d- C_2),$$
for all $h\ge \max\{1,\ (k-1)(d - 1)d \}$.
\end{Theorem}

In order to get information on the smallest possible integer $h$  in the above theorem, the following notation was recently introduced.

\begin{Definition}
{\rm \cite[Definition 1.3]{GS} \label{Sumreg} The least integer $\sigma$ such that the decomposition in Theorem \ref{NaTh11} holds for all $h\ge \sigma$ is called the {\it sumsets regularity} of  $A$  and  denoted  by $\sigma(A)$.}
 \end{Definition}
 
 It turns out that there is a close relationship between the above invariant and other invariants of the ring $K[S]$ of the monomial curve $C(a_1,\ldots,a_k)$, see, e.g.,  \cite{El}. Below we recall a relationship between the sumset regularity and the Castelnuovo--Mumford regularity $\reg(K[S])$.  
 Let $H_{K[S]}(n) = \dim_K K[S]_n$  be the Hilbert function of $K[S]$. Then $H_{K[S]} (n)$ agrees with $P_{K[S]}(n)$ for all $n\gg 0$, where $P_{K[S]}(t)$ is the so-called  Hilbert polynomial. The {\it regularity index} (of Hilbert function $H_{K[S]} (n)$) is the number
 $$\ri (S) := \min\{ n_0 \geq 0\mid H_{K[S]} (n) = P_{K[S]} (n) \ \text{for all}\ n\geq n_0 \}.$$
 Since $K[S]$ has positive depth, from the Grothendieck--Serre  formula
$$ P_{K[S]}(t) - H_{K[S]}(t) =  \ell(H^1_{\mfr}(K[S])_t -  \ell(H^2_{\mfr}(K[S])_t $$
we always have $\ri(S) \leq \reg(K[S])$ with equality when $a_1(K[S]) >  a_2(K[S])$. With the remark that $F(A) + 1$ is the conductor of $\langle A\rangle$, we have 
 
 \begin{Lemma}\label{GSTh31}
 \begin{enumerate}[\((i)\)]
     \item {\rm \cite[Theorem 3.1]{GS}} We have
$$\sigma (A) = \max\left\{ \ri (K[S]),\ \left\lceil \frac{F_1 + F_2 + 2}{d} \right\rceil \right\}.$$
\item {\rm \cite[Theorem 3.16]{GS}} $\reg(K[S]) \leq \sigma(A)+1$. More precisely,
\begin{enumerate}[\((a)\)]
    \item If $\ri(S) \geq \left\lceil \frac{F_1 + F_2 + 2}{d} \right\rceil$, then $\sigma(A) \leq \reg(K[S]) \leq \sigma(A)+1$.
    \item If $\ri(S) < \left\lceil \frac{F_1 + F_2 + 2}{d} \right\rceil$, 
 then $\left\lceil \frac{\sigma(A)  }{2} \right\rceil + 1 \leq \reg(K[S]) \leq \sigma(A)+1$.
\end{enumerate}
 \end{enumerate}
\end{Lemma}
 
 In some particular cases, there are  more precise relationships between the invariants involved in the above theorem, see Propositions 3.2 and 3.18, and Corollary 3.21 in \cite{GS}. Below is another case.
 
 \begin{Corollary} \label{GSCor}  Assume that $a_1 = 1$. Then $\sigma(A) \leq \reg(K[S]) \leq \sigma(A)+1$.
\end{Corollary}
 
 \begin{proof} The upper bound is Lemma \ref{GSTh31}\((ii)\). For the lower bound, note that $F_1 = -1$, because   $a_1= 1$. Choose an index $0< i\leq d-1$ such that $\omega_2(d-i) = F_2 + d$. By Lemma \ref{B1}, 
 $$\reg(K[S]) \geq a_2(K[S]) + 2 \geq \deg(\wbf_i) = \frac{F_2+d + \omega_1(i)}{d} \geq \frac {F_2+1}{d} + 1,$$
 whence
 $\reg(K[S]) \geq \left\lceil \frac{F_2+1}{d} \right\rceil +1 > 
 \frac{F_1+F_2 +2}{d}$. Since $\ri(S) \leq \reg(K[S])$, by Lemma \ref{GSTh31}\((i)\), $\reg(K[S]) \geq \sigma(A)$, as required.
\end{proof}
Examples in \cite[Table 1]{GS} show that both cases in the above corollary can happen. Namely, $\reg(K[S]) = \sigma(A)$ for the curve $C(1,3,11,13)$, and $\reg(K[S]) = \sigma(A) +1$ for the curve $C(1,3,5,6,12)$.

 The best bound on $\sigma(A)$ until now  is given in \cite[Theorem 1]{GW} as follows:
 $$\sigma (A) \leq  d - k +1 = \deg(C) - \codim (C).$$
 This upper bound is exactly the well-known Gruson--Lazarsfeld--Peskine bound for curves. In \cite[Section 5]{GS}, Gimenez and Gonz\'alez-S\'anchez showed that one can deduce the Granville--Walker bound from the one of Gruson--Lazarsfeld--Peskine, and vice versa. In particular,  \cite[Theorem 3.1]{GS} provides another combinatorial proof of the Gruson--Lazarsfeld--Peskine bound. Note that the first combinatorial proof of the Gruson--Lazarsfeld--Peskine bound was given in \cite{Nit}.
 
 Below we show that $\sigma(A)$ is also bounded by L'vovsky's bound.
 
 \begin{Proposition} \label{SumsetLv} L'vovsky's bound holds for $\sigma(A)$, that is
 $$\sigma(A) \leq \lambda_{\max} + \lambda_{\mathrm{sl}} + 1.$$
\end{Proposition}
 
 \begin{proof} By Schur's bound (\ref{EFRb1}), we have
 \begin{align*}
     \left\lceil \frac{F_1+F_2 + 2}{d} \right\rceil&\leq \left\lceil \frac{(a_1-1) (d-1) + (d-a_{k-1}-1)(d-1)  }{d} \right\rceil \\
 &\leq  (a_1 -1) + (d-a_{k-1} -1) \leq \lambda_{\max} + \lambda_{\mathrm{sl}}.
 \end{align*}
On the other hand, by L'vovsky's bound, $\ri(S) \leq \reg(K[S]) \leq \lambda_{\max} + \lambda_{\mathrm{sl}} + 1$. Hence the statement immediately follows from Lemma \ref{GSTh31}\((i)\).
\end{proof}
Note that using the results in Section \ref{Breg}, under some specific conditions,  the proof of the above proposition also gives several better bounds on $\sigma(A)$ than L'vovsky's bound. An open question is if one can find a combinatorial proof  for the above proposition (i.e. without using L'vovsky's result).

\vskip0.5cm
\noindent {\bf Acknowledgement}.
 The first author is partially supported by the Research Project NAFOSTED 101.04-2024.07.

\section*{APPENDIX}

In this section we construct several algorithms.

In the first algorithm, given a subset $\Acal$ we compute the Ap\'ery set $Ap(\Acal)$ as a vector in $ \Nset^{d}$ and degrees of its elements as a vector $\Deg(\Acal) \in \Nset^{d}$. Given an integer $a$ and a positive integer $d$, we denote by $\rem(a;d)$ the reminder of $a$ on division by $d$.
\vskip0.5cm

\noindent {\bf Algorithm 1}: 
Computation of  Ap\'ery set $Ap(\Acal)$ and its degree vector $D(\Acal)$\\

\begin{tabular}{ll}
Input: &  $ \Acal := \{\alpha_1,\ldots,\alpha_k =: d\}$    \\
Output: & $ Ap(\Acal)  = (\omega(0), \omega(1),\ldots, \omega(d-1) ) $ \\
& $ D(\Acal) = (\delta (\omega(0)), \delta(\omega(1)),\ldots, \delta(\omega(d-1) ))$ \\
Initial: &  $n:= 1; \ \Acal' :=  \Acal\setminus \{d\};\  \Omega :=  \Acal'; \ Ap(0):= 0;\ D(0) := 0$;  \\
& WHILE $1\leq i\leq d-1$ DO \\
& \hskip0.8cm IF $ i\in \Acal$ THEN $Ap(i) := i;  D(i):= 1$ ELSE $ Ap(i) := i +d\alpha_1;\ D(i) := d. $\\
 REPEAT &  \\
&  \hskip1cm $n:= n+1; \ \Omega1 := \emptyset; \ C:= \Acal'\times \Omega$ \\
&  \hskip1cm FOR each $(\alpha, y) \in C$ \\
& \hskip2cm $i:= \rem(\alpha + y; d)$\\
& \hskip2cm IF $\alpha+y < Ap (i)$ THEN\\
& \hskip3cm $Ap(i) := \alpha + y;\ D(i) := n;\ \Omega1 := \Omega1 \cup \{\alpha + y\}$ \\
&  \hskip1cm $\Omega:= \Omega1$\\
 UNTIL & $\Omega = \emptyset$
\end{tabular} 
\vskip0.5cm

In the second algorithm, given the set $A_1$ we compute the set $\Iset$ in Lemma \ref{A3} and $a_2(K[S])$.  We also compute the set 
$\Wbf(A_1) =\{\wbf_0,\wbf_1,\ldots,\wbf_{d-1} \}$,  the degree vectors  $\Deg(A_1), \ D1(A_1), \ D2(A_1) \in \Nset^{d}$ of elements of $\Wbf(A_1)$ and their components.
\vskip0.5cm

\noindent {\bf Algorithm 2}:  Computation of $\Iset$ and $a_2(K[S])$

\begin{tabular}{l}
Input:   $ A1:= \{a_1,\ldots,a_k =: d\}$    \\
Output:  $\Iset (A1) := \Iset$; $ a(A1):= a2 = a_2(K[S])$;  
$\Wbf (A1) := \Wbf =  (\wbf_1,\ldots,\wbf_{d-1})$, \\
\hskip1.6cm  $\Deg(A1) :=  \Deg = (\deg(\wbf_1),\ldots, \deg(\wbf_{d-1}))$,\\
\hskip1.6cm $ D1(A1) := D1  = (\delta_1(\omega_1(1)), \delta_1(\omega_1(2)),\ldots,\delta_1(\omega_1(d-1)) )$ \\
 \hskip1.6cm $ D2(A1) := D2 = (\delta_2(\omega_2(d-1)), \delta_2(\omega_2(d-2)),\ldots,\delta_2(\omega_2(1))$\\
 Initial: $\Iset := \emptyset$ \\
 $A2:= \{d-a_{k-1},\ldots,d-a_1, d\}$ \\
 \text{\it \%  Computation by Algorithm} 1\\
 $Ap1 := Ap(A1);\  Ap2 := Ap(A2); D1 := D(A1); D:= D(A2) $\\
 WHILE $1\leq i\leq d-1$ DO \\
 \hskip1cm $\wbf_i := ( Ap1(i), Ap2(d-i) );\ \Deg(i) := ( Ap1(i)+ Ap2(d-i))/d;\ D2(i) := D(d-i)$\\
   \hskip1cm  IF $D1(i) > \Deg(i)$ THEN $\Iset := \Iset \cup \{ i\}$\\
   $a2 := \max\{\Deg(1),\ldots,\Deg(d-1)\} - 2$.  
\end{tabular} 
\vskip0.5cm

In the last algorithm, we check if $K[S]$ is a Cohen--Macaulay or Buchsbaum ring, respectively, and compute the Castelnuovo--Mumford regularity as well as the length of the first local cohomology module. 
\vskip0.5cm

\noindent {\bf Algorithm 3}: Computation of  $\reg(K[S])$ and $\ell(H^1_{\mfr}(K[S]))$

\begin{tabular}{l}
Input: $A := \{a_1,\ldots,a_k =:d\}$\\
Output: $a_1(K[S]) = a1$, $ \reg (K[S]) = reg$, $ \sharp(\Lcal) = L$; \ $\ell(H^1_{\mfr}(K[S])) = \ell$ \\
\hskip2cm $ CM$ ({\it = 1 if $K[S]$ is Cohen--Macaulay; 0: otherwise})\\ 
\hskip2cm $Bbm$ ({\it = 1 if $K[S]$ is Buchsbaum; 0: otherwise})\\
$Ap1:= Ap(A)$ (\text{\it  \% By Algorithm 1} )\\
\text{\it \% Computation by Algorithm 2} \\
$\Wbf := \Wbf(A); \Deg := \Deg(A);\ D1:= D1(A);\ D2 := D2(A)$\\
$ \Iset := \Iset(A)$; $a2 := a2(A)$; $reg:= a2 + 2$\\
IF $\Iset = \emptyset$ THEN  $CM:= 1$  and  Goto STOP\\
ELSE \\

\hskip1cm $CM := 0$; $L := \sharp(\Iset)$; 
$ND2:= \max\{ D2(i) - \Deg(i) -1\mid i\in \Iset\}$\\
\hskip1cm IF $ND2=0$ THEN \\
\hskip2cm  $ND1 := \max\{ D1(i) - \Deg(i) -1\mid
i\in \Iset\}$\\
\hskip2cm IF $ND1 =0$ THEN \\
\hskip3cm $a1 := \max\{ \Deg(i)\mid i\in \Iset \}$;  $reg :=\max\{ a1+1; \ reg \}$; $\ell := L$ \\
\hskip3cm FOR each $i, j\in \Iset$ and $\wbf(i) > \wbf(j) $\\
\hskip4cm IF $\wbf(i) - \wbf(j) \in \{\ebf_h\mid h \in A\}$ THEN\\ \hskip5cm $Bbm:= 0$ and Goto  STOP \\
\hskip3cm  $Bbm:= 1$ and Goto STOP\\
\hskip2cm ELSE\\
\hskip3cm  $Bbm:= 0$; $a1:= \max\{D1(i)\mid i\in \Iset \} -1$\\
\hskip3cm  $\reg := \max\{a1+1,\ reg\}$; $\ell:= 0$\\
\hskip3cm FOR each  $i\in \Iset$ DO
  $\ell := \ell + D1(i) - \Deg(i)$\\
\hskip3cm Goto STOP\\
\hskip1cm  ELSE\\
\text{\it \% Computation of degree of $\langle A_1\rangle$ upto degree $N1+ND2$}\\
\hskip2cm  $Bbm:= 0$; $n:= 1; \ An:= A; N1:= \max\{D1(i) \mid i\in \Iset\}$   \\
\hskip2cm  $TotD (i,j):= \begin{cases} 0 & \text{if}\  i= 0, \  \text{and}\ j=0\\
1 & \text{if}\ i\in A \  \text{and}\ j=0,\\
2d^2 & \text{if} \ 1\leq j \leq N1 + ND2.
\end{cases}$\\
\hskip2cm REPEAT   \\
 \hskip3cm $n:= n+1;  \ B:= \emptyset;  \ C:= An \times A$ \\
\hskip3cm FOR each $(\alpha, y) \in C$ \\
 \hskip4cm $i:= \rem(\alpha + y; d);\ j := (\alpha+y - i)/d$\\
 \hskip4cm IF $n < TotD(i;j)$ THEN\\
 \hskip5cm $TotD(i;j) := n;\ B := B \cup \{\alpha + y\}$ \\
  \hskip3cm $An:= B$\\
 \hskip2cm UNTIL $n> N1+ND2$ \\
 $a1:= 0$; \ $\ell:= 0$\\
 WHILE $i\in \Iset$\\ 
 \hskip1cm $m:= 0; ND2i := D2(i) - \Deg(i) -1$; $L := L + ND2i$\\
 \hskip1cm WHILE $m \leq ND2i$ DO\\
 \hskip2cm $j:= ( Ap1(i)-i)/d ); \  n2:= TotD(i;j) - \Deg(i) -1$\\
 \hskip2cm $\ell:= \ell + n2 + 1$; $a1:= \max\{a1, \Deg(i) + m + n2\}$\\
 $reg:= \max\{a_1+1;\ reg\}$
\end{tabular}
\vskip0.5cm

\begin{proof} 
\begin{enumerate}[\((a)\)]
    \item \textit{Correctness of Algorithm 1.} By Schur's bound (\ref{EFRb1}), for any $i$, $\omega(i) <  d\alpha_1$ and $\delta(\omega(i)) < d$. Hence,  when $i\not\in \Acal$, setting $Ap(i) := i +d\alpha_1;\ D(i) := d$ is equivalent to saying that $\omega(i)$ and its degree  are still not defined.

Denote by $Ap_n,\ \Omega_n, \ D_n$ the resulted sets $Ap,\ \Omega$ and vector $D$ in Step $n$, respectively. Recall that $\Acal' = \Acal \setminus \{d\}$. Then by induction, it is clear that $\Omega_n \subset n\Acal'$, whence $\delta(x) \leq n$ for any $x\in \Omega_n$.

We show by induction on $ m = \delta(\omega(i))$ that $\omega(i) \in \Omega_m$, $D_m(i) = \delta(\omega(i))$  and  $Ap_n \neq \emptyset$ for all $n\leq m$. By the initial step, this holds for $m=1$. Assume that $\delta(\omega(j)) = m >1$ and the claim holds true for all $\omega(i)$ with $\delta(\omega(i)) \leq m-1$. By the definition of degree we can find $\beta_1,\ldots,\beta_m \in \Acal'$ such that
$\omega(j) = \beta'  + \beta_m$, where $\beta' = \beta_1 + \cdots + \beta_{m-1}$. By Remark \ref{degree}, we must have $\delta(\beta') = m-1$. Assume that $\beta' = \omega(h) + dq$ for some $h\leq d-1$ and $q\in \Nset$. Then $\omega(h) + \beta_m$ and $\omega(j)$ belong to the same residue class $j$ modulo $d$. From the minimality of $\omega(j)$ we must have $q=0$, and $\delta(\omega(h)) = \delta(\beta') = m-1$. By induction hypothesis, $\beta' = \omega(h) \in \Omega_{m-1}$, whence  $\Omega_n \neq \emptyset$ for all $n\leq m-1$.

Since $\delta(\omega(j)) = m$, $\omega(j)$ cannot appear in the previous steps. Since $\omega(j) = \beta'+\beta_m$ is the smallest number in the residue class $j$, in this step we must have $\beta'+\beta_m < Ap_{m-1}(j)$. Hence,  $Ap_m(j)$ now takes the value $\omega(j)$, $D_m(j) = \delta(\omega(j))$  and $\omega(j) \in \Omega1 \subset \Omega_m$.
\item \textit{Termination of Algorithm 1.} Once all $\omega(1),\ldots,\omega(d-1)$ have been appeared before Step $n$,  no pair $(\alpha, y) \in \Acal' \times Ap_n$ satisfies $\alpha + y < \omega(i)$. This means $\Omega1 = \emptyset$, whence $\Omega_n = \emptyset$, and the algorithm terminates. By the above argument, it terminates exactly at the step $N := \max\{ \delta(\omega(1)),\ldots,\delta(\omega(d-1)\} + 1$.

When the algorithm terminates (at Step $N$), we see that for any $i\leq d-1$ we have  $Ap(i) = \omega(i)$ and $D(i) = \delta(\omega(i))$. That means the output gives us the correct answer.
\item \textit{Correctness of Algorithm 2.} follows from \((a)\) above,  the definition of  the set $\Iset$ in Lemma \ref{A3} and Lemma \ref{B1}.
\item \textit{Correctness of Algorithm 3.} Note that $A = A_1\setminus \{0\}$. We use Theorem \ref{A4a} and Theorem \ref{A3New}. All variables there can be computed, using Algorithm 2, except $\delta_1(\omega_1(i) + n_1d)$ in the formula there. It is enough to show that the algorithm computes this value if needed.

Since $\delta_1(\omega_1(i) + n_1d) \leq \delta_1(\omega_1(i)) + n_1 \leq N1+ND2 =: M$ (by Remark \ref{degree}), where $N1,\ ND2$ are defined in Algorithm 3, it is enough  to compute degrees of all elements of $\cup_{n=1}^{M} (nA)$. 
Let $A(n):= nA \setminus \cup_{i=1}^{n-1} (iA)$. Then for any $0< x\in \langle A\rangle$, $\delta_1(x) =n$ if and only if $x\in A(n)$. If $x = \beta_1 + \cdots + \beta_n$, where $\beta_i \in A$, let $y:= \beta_1 + \cdots + \beta_{n-1}$. Then by Remark \ref{degree}, we must have $\delta_1(y) = n-1$, whence $A(n) \subseteq A(n-1) + A$. 

On the other hand $TotD(i,j)$ is the degree $\delta_1(i + jd)$ if $i+jd\in \langle A \rangle$ and $\infty$ otherwise. As showed in \((a)\), $M= N1 + ND2 < 2d$, all elements of $\cup_{n=1}^{M} (nA)$ are less than $2d^2$. This means the assignment $TotD(i,j) = 2d^2$ in the initial step is equivalent to saying that either $i+jd \not\in \langle A \rangle$ or its degree is still not defined. Note that before a possible reassignment at the step $n$, either $TotD(i,j) = 2d^2$ or $TotD(i,j) < n$

We denote the resulted set $An$ at the step $i$ by $A_i$.  By induction we now show that $A_n = A(n)$ and $TotD(i,j) = n$ for all $i+jd \in A_n$. This is clear for $n=1$. In the step $(n+1)$ we have $C= A_n \times A$. Fix a pair  $(\alpha, y)\in C$, let  $i+jd = \alpha + y \in A(m)$, where $0\leq i\leq d-1$, and $m= \delta_1(\alpha + y)$. Then $m\leq n +1$. If $m< n+1$,  by induction we already assign $TotD(i,j) = m < n+1$, so $TotD(i,j)$ remains unchanged in this case. Moreover $ \alpha + y \not\in B$

If $m= n+1$, then by remark above, $TotD(i,j) = 2d^2 >m$. Therefore, by the algorithm, we now assign $TotD(i,j) := m $ and add $\alpha+y$ to $B$. So, at the end of Step $(n+1)$, we have $B= A(n+1)$, whence $A_{n+1}= B = A(n+1)$, as required.
\end{enumerate}
\end{proof}

\begin{Remark}
    The above algorithms can be run using the link \url{https://sites.google.com/view/quangtiendoan/home} on the website of the second author.
\end{Remark}

\

{
\noindent Le Tuan Hoa

\noindent Address: \emph{Institute of Mathematics, VAST, 18 Hoang Quoc Viet, 10307 Hanoi, Vietnam}

\noindent  Email: \emph{lthoa@math.ac.vn}

\

\noindent Doan Quang Tien

\noindent Address: \emph{Institute of Mathematics, VAST, 18 Hoang Quoc Viet, 10307 Hanoi, Vietnam}

\noindent Email: \emph{doanquangtien1442001@gmail.com}

\end{document}